\documentclass[12pt]{amsart}
\usepackage{amsthm,amssymb,verbatim,amsmath, amsfonts, amscd
   ,a4wide
}
\usepackage{ascmac}    

\usepackage{hyperref}              
\usepackage{verbatim} 
\usepackage{hyperref}

\theoremstyle{plain}\newtheorem{Theorem}{Theorem}[section]
\theoremstyle{plain}
\theoremstyle{plain}
\theoremstyle{plain}\newtheorem{Lemma}[Theorem]{Lemma}
\theoremstyle{plain}
\theoremstyle{definition}
\theoremstyle{definition}
\theoremstyle{definition}
\theoremstyle{definition}\newtheorem{Remark}[Theorem]{Remark}
\theoremstyle{definition}
\theoremstyle{definition}\newtheorem{Notation}[Theorem]{Notation}
\theoremstyle{definition}\newtheorem{Hypothesis}[Theorem]{Hypothesis}
\theoremstyle{definition}
\theoremstyle{definition}
\theoremstyle{definition}
\theoremstyle{definition}
\theoremstyle{definition}
\theoremstyle{definition}\newtheorem{Notation/Definition}
[Theorem]{Notation/Definition}
\theoremstyle{definition}

\def\dim{\mathrm{dim}}

\def\ker{\mathrm{ker}}

\begin{document}
\title{The Principal $p$-blocks with small numbers of characters
}
\date{\today}
\author{\bf Shigeo Koshitani$^*$, 
Taro Sakurai}  
\address{Center for Frontier Science,
Chiba University, 1-33 Yayoi-cho, Inage-ku, Chiba 263-8522, Japan.}
\email{koshitan@math.s.chiba-u.ac.jp}
\address{ Department of Mathematics and Informatics, Graduate School of Science,
Chiba University, 1-33 Yayoi-cho, Inage-ku, Chiba 263-8522, Japan.}
\email{tsakurai@math.s.chiba-u.ac.jp}

\footnotetext[0]{$^*$Corresponding author}

\thanks{The first author was partially supported by 
the Japan Society for Promotion of Science (JSPS), Grant-in-Aid for Scientific Research
(C)19K03416, 2019--2021}

\maketitle

\begin{abstract}
For a prime $p$, we determine a Sylow $p$-subgroup $D$ of a finite group $G$ such that
the principal $p$-block $B$ of $G$ has four irreducible ordinary characters.
It has been determined already for the cases where the number is up to three by work by R.~Brauer, J.~Brandt, and V.A.~Belonogov thirty years ago.
Our proof relies on the classification of finite simple groups.
\end{abstract}
\maketitle

\section{Introduction}

As Richard Brauer said, in $p$-modular representation theory of finite groups, where $p$ is a prime, 
one of the most important and interesting problems is to 
determine the number $k(B)$ of 
irreducible ordinary characters of a finite group $G$ in a $p$-block $B$ of $G$
if a defect group $D$ of $B$ is given
(see \cite[Problem 20]{Bra63}).
In particular, we want to know what we can say about the structure of $B$ if $k(B)$ is small.
It is known for instance that $k(B)=1$ if and only if $D$ is trivial, namely, 
the block algebra $B$ is a simple algebra.
Further we know that $k(B)=2$ if and only if the order of $D$ is two
(see \cite[Theorem A]{Bra82}).
Then, it is quite natural to wonder what happens if $k(B)=3$.
It is kind of surprising that we do not know the answer yet
in spite of the two celebrated results about
the half of Brauer's Height Zero Conjecture
\cite{KM13} and
McKay's Conjecture for $p=2$ \cite{MS16}. 
Going back to the question that $k(B)=3$ for an arbitrary $p$-block $B$,
we know that if $k(B)=3$ and $\ell(B)=1$,
then the order of $D$ is three by
\cite[Theorems A and  B]{Bra82} and \cite{Dad66}, 
where $\ell(B)$ is the number of irreducible Brauer characters
of $G$ belonging to $B$.  
So what is left is to determine $D$ if $k(B)=3$ and $\ell(B)=2$.
As far as we know, this is still open though we guess that $D$ should have order three.
Actually only for the particular case that $B$ is the principal $p$-block, we have an answer, namely
in \cite{Bel90} V.A.~Belonogov proves that if $B$ is the principal $p$-block and $k(B)=3$, then $D$ has
order three.

The next natural situation
is surely looking at the case where
$k(B)=4$. Actually, this is the case of our main result in this paper.
Namely our object is to prove the following:

\begin{Theorem}\label{k=4MainTheorem}
Let $B$ be the principal $p$-block of a finite group $G$ with a Sylow $p$-subgroup $D$. Further, let $k(B)$ and $\ell(B)$ respectively be the numbers of
irreducible ordinary and Brauer characters of $G$ belonging to $B$.
Assume that $k(B)=4$. Then the following holds
(depending on the classification of finite simple groups)
:
\begin{enumerate}
\item[\rm (i)] If  furthermore $\ell(B)=2$, then $D\cong C_5$ (the cyclic group of order $5$).
\item[\rm (ii)] If furthermore $\ell(B)=3$, then $D\cong C_2\times C_2$ (the Klein four group).
\end{enumerate}
\end{Theorem}

\begin{Remark}
It is well-known that $G$ is $p$-nilpotent (and hence the principal $p$-block
is nilpotent) if $\ell(B)=1$. More generally, under the assumption that $\ell(B)=1$
for an arbitrary $p$-block $B$ of $G$, there are interesting results
due to K{\"u}lshammer and Chlebowitz 
in \cite{Kue84, CK92}.
There is also a related result \cite{KS19} where small blocks of (not necessarily
of finite groups)
finite dimensional $F$-algebras are investigated.
\end{Remark}

\begin{Notation}
Throughout this paper $G$ is always a finite group, $p$ is a prime,
$F$ is an algebraically closed field of characteristic $p>0$, and $FG$ is the group
algebra of $G$ over $F$.
We usually denote by $q$ a power of a certain prime unless stated otherwise.
For such a $q$ we write $\mathbb F_q$ for the finite field of $q$ elements
and $\overline{\mathbb F}_q$ for the algebraic closure of  $\mathbb F_q$.
We denote by $G_{p'}$ the set of all $p$-regular elements of $G$,
by $G^\#$ the set $G-\{1\}$, by ${\mathrm{exp}}(G)$ the exponent of $G$,
and  by $Z(G)$ the center of $G$.
We write $H \ {\mathrm{char}}\ G$ when $H$ is a characteristic subgroup of $G$.
We write $\mathbb N$ for the set of all positive integers.
For any $n\in\mathbb N$, $C_n$, $D_n$, and $Q_n$ respectively are
the cyclic group, the dihedral group, and the quaternion group of order $n$.
For any $n\in\mathbb N$ we denote by $\mathfrak S_n$ and $\mathfrak A_n$ respectively
the symmetric and the alternating groups of degree $n$.
For $g\in G$, $K_g$ denotes the conjugacy class of $G$ which contains $g$.
We write $O_p(G)$ and $O_{p'}(G)$ respectively for 
the largest normal $p$-subgroup of $G$
and the largest normal $p'$-subgroup of $G$.
We write similarly $O^p(G)$ and $O^{p'}(G)$ respectively for the smallest
normal subgroup $N$ of $G$ such that $G/N$ is a $p$-group and
the smallest normal subgroup $L$ of $G$ such that $G/L$ is a $p'$-group. 
We write ${\mathrm{Syl}}_p(G)$ for the set of all Sylow $p$-subgroups of $G$.
For two groups $H$ and $L$, we denote by $H\rtimes L$ a semi-direct product of $H$
by $L$, namely $H\vartriangleleft H\rtimes L$.
For two subgroups $H$ and $L$ of $G$, we write $H=_G L$ if there is an element $g\in G$
such that $L = g^{-1}Hg$.
For elements $x, g\in G$, we define $x^g:=g^{-1}xg$.
We write $k(G)$ and $\ell(G)$ for the numbers of the conjugacy classes and
the $p'$-conjugacy classes of $G$, respectively.
We denote by ${\mathrm{Irr}}(G)$ the set of all irreducible ordinary characters of $G$,
and it is well-known that $|{\mathrm{Irr}}(G)|=k(G)$. 
For a $p$-subgroup $P$ of $G$ we write $\mathcal F_P(G)$ for the fusion system (fusion category)
of $G$ over $P$ (see \cite[I, Part I]{AKO11}).
    
Let $B$ be a block (algebra) of $FG$, and let $1_B$ be the block idempotent of $B$.
As usual, we denote by ${\mathrm{Irr}}(B)$ and ${\mathrm{IBr}}(B)$ respectively
the sets of all irreducible ordinary and Brauer characters of $G$ belonging to $B$,
and we denote by $k(B)$ and $\ell (B)$ respectively  $|{\mathrm{Irr}}(B)|$ and
$|{\mathrm{IBr}}(B)|$.
We denote by $k_0(B)$ the number of elements in ${\mathrm{Irr}}(B)$ of height zero.
We write $Z^{\mathrm{pr}}(B)$ for the projective center of $B$ and
$\bar Z(B)$ the stable center of $B$, that is 
$\bar Z(B):= Z(B)/Z^{\mathrm{pr}}(B)$ (see \cite[pp.4 and 20]{Bro94} and \cite[p.127]{KL02}).

For ordinary characters $\chi$ and $\psi$ of $G$, we denote by $\langle \chi, \psi\rangle^G$
the (usual) inner product.
For an ordinary character $\psi$ of $G$ we denote by ${\mathrm{Irr}}(G|\psi)$ the set of all
elements $\chi\in{\mathrm{Irr}}(G)$ such that $\langle\psi, \chi\rangle^G\,{\not=}\,0$,
and we define ${\mathrm{Irr}}(B|\psi):={\mathrm{Irr}}(G|\psi)\cap{\mathrm{Irr}}(B)$.
Further, let $C_B$
be the Cartan matrix of $B$.
We write $1_G$ for the trivial ordinary character of $G$.
For an ordinary character $\chi$ of $G$, we write $\chi^0$ for the restriction of $\chi$
to $G_{p'}$. 
We write $B_0(FG)$ for the principal block (algebra) of $FG$,
and $F_G$ for the trivial (right) $FG$-module.
For a subgroup $H$ of $G$ and for characters $\chi$ of $G$ and $\theta$ of $H$,
we denote by $\chi{\downarrow}_H$ and $\theta{\uparrow}^G$ the restriction of $\chi$ to $H$
and the induction (induced character) of $\theta$ to $G$, respectively.
Let $N\vartriangleleft G$.
For an ordinary character $\theta$ of $N$ and an element $g\in G$
we define $\theta^g$ by $\theta^g(n)=\theta (gng^{-1})$ for $n\in N$.
Let $G_\theta$ be the stabilizer (inertial subgroup) 
of $\theta$ in $G$, namely $G_\theta:=\{g\in G\,|\,\theta^g=\theta\}$.

Let $A$ be a finite dimensional $F$-algebra. We denote by $Z(A)$ the center,
by $J(A)$ the Jacobson radical,
and by $LL(A)$ the Loewy length of $A$, respectively.
    
For the other notation and terminology, see the books \cite{NT89} and \cite{Gor68}.
\end{Notation}

This paper is organized as follows. In \S2 we shall give general lemmas
that are useful.
In \S3 we shall reprove Theorem \ref{MainTheorem},
namely we shall investigate the principal
$p$-blocks $B$ such that $k(B)=3$ and $\ell(B)=2$.
In \S4 we shall give lemmas which shall be used in \S5.
In \S5 we shall prove our first main result, namely, we shall investigate the principal
$p$-blocks $B$ such that $k(B)=4$ and $\ell(B)=3$.
Finally in \S6 we shall prove our second main result, that is, we shall investigate
the principal $p$-blocks $B$ with $k(B)=4$ and $\ell(B)=2$.

\section{Preliminaries}

In this section we shall list several lemmas which shall be
used for the proof of our main results.
The next lemma is quite useful for our aim.

\begin{Lemma}\label{KL02-p.143}
Let $B$ be a block of $FG$ with defect group $D$.
\begin{enumerate}
\item[(i)]
$\dim_F Z^{\mathrm{pr}}(B)$ is the multiplicity of $1$ as elementary divisor of 
$C_B$. 
\item[(ii)]
If further
$C_G(u)$ is $p$-nilpotent for every element $u\in D$ of order $p$, then
$$ \dim_F {\bar Z}(B) = k(B)-\ell(B)+1.$$
\end{enumerate}
\end{Lemma}

\begin{proof} (i) See \cite[p.4]{Bro94} or \cite[p.143, lines $-7\sim -6$]{KL02}.

(ii) By \cite{Fuj80}, $\det(C_B)=|D|$. Hence, \cite[Theorem 3.6.35]{NT89} implies that
the elementary divisors of $C_B$ are $|D|$ with multiplicity one and $1$ with multiplicity $\ell(B)-1$. 
Hence by (i), $\dim_F Z^{\mathrm{pr}}(B)=\ell(B)-1$. Therefore,
$\dim_F\bar Z(B)=\dim_F Z(B)-\dim_F Z^{\mathrm{pr}}(B) = k(B)-(\ell(B)-1)$
since $k(B)=\dim_FZ(B)$ (see \cite[Lemma 5.11.3]{NT89}).
\end{proof}

\begin{Lemma}[Navarro-Tiep]\label{HZCp=2}
Assume that $p=2$ and $B:=B_0(FG)$.
If $k(B)=4$, then a Sylow $2$-subgroup of $G$ is abelian.
\end{Lemma}

\begin{proof}
Follows by \cite[Corollary 1.4(i)]{Lan81}
and \cite[Theorem A]{NT12}
(note that \cite{NT12} uses the {\sf CFSG} and does not \cite{Lan81}).
\end{proof}

\begin{Lemma}\label{[KNST14,Theorem A]} 
Assume that a pair $(p,G)$ is one of the pairs of 
the primes $p$ and the finite simple groups $G$
listed in \cite[(ii) and (iii) in Theorem A]{KNST14}, and let $B:=B_0(FG)$. 
Then $k(B)-\ell(B)\,{\not=}\,2$ and ${\not=}\,1$
\end{Lemma}

\begin{proof}
Follows from  \cite[(ii) and (iii) in Theorem A]{KNST14} and \cite[Tables 1 and 2]{Hen07}.
\end{proof}

\begin{Lemma}[\cite{BCR90}]\label{ZJ_new}
Assume that 
$p$ is odd, and that
for every element $u\in G$ of order $p$, 
\linebreak
$C_G(u)$ is $p$-nilpotent.
Set $H:= N_G\Big(Z(J(D))\Big)$ where $D\in{\mathrm{Syl}}_p(G)$ and 
$J(D)$ is the Thompson subgroup of $D$.
Then
$p\,{\not{\mid}}\,|G:H|$, 
and ${\mathrm{H}}^*(G,F)\cong{\mathrm{H}}^*(H,F)$.
\end{Lemma}

\begin{proof}
Follows by \cite[Proposition 9.2 and lines $-7\sim -4$, p.62]{BCR90}.
\end{proof}

\noindent
It seems that the following lemma demands too strong assumption, however it
will be quite useful for our aim because of the stable equivalence of Morita type 
due to M.~Brou\'e.

\begin{Lemma}\label{k-ell=2}
Suppose that $D\in{\mathrm{Syl}}_p(G)$ with $D\vartriangleleft G$ and that
$G=D\rtimes E$ for a $p'$-subgroup $E$ of $G$ such that $E$ acts on $D$ faithfully
(and hence $O_{p'}(G)=1$, so that $B:=B_0(FG)$ is the unique block of $FG$
by \cite{FG61}). Assume furthermore that $k(B)-\ell(B)=2$ and that
$C_G(u)$ is $p$-nilpotent for every element $u\in D$ of order $p$.
Then it holds the following:
\begin{enumerate}
\item[(i)] Every $p'$-conjugacy class $K$ of $G$ with $K\,{\not=}\,\{1\}$ is of defect zero.
\item[(ii)] $E$ acts on $D^\#$ semi-regularly (i.e. freely).
\item[(iii)] Let $K_0:=\{1\}, K_1, \cdots, K_{\ell(G)-1}$ be all $p'$-conjugacy classes of $G$. Then,
$$\bigsqcup_{i=1}^{\ell(G)-1} K_i = \{vy \,|\, v\in D, y\in E^\# \}=DE^\#.$$
\item[(iv)] $D$ is abelian.
\item[(v)] There are elements $u, v\in D^\#$ such that
$D^\#=K_u\,\sqcup\, K_v 
\text{ (disjoint union)}$ and
$|K_u|=|K_v|=|E|$.
\item[(vi)] The projective center $\bar Z(B)$ is a symmetric $F$-algebra of 
$F$-dimension $3$.
\end{enumerate}
\end{Lemma}

\begin{proof}
(i) 
It follows from the assumption that $C_G(u)$ is $p$-nilpotent for every $u\in G$ of order $p$
and \cite{Fuj80} that $\det(C_B)=|D|$. Hence \cite[Theorem 3.6.35]{NT89} implies that all the elementary
divisors of $C_B$ is $|D|$ with multiplicity one and $1$ with multiplicity $\ell(B)-1$. 
Thus the assertion follows from \cite[Theorem 5.11.6(iv)]{NT89}.

(ii) 
Apparently, $E$ acts on $D^\#$. Take any element $v\in D^\#$ and any $p'$-element $y\in C_E(v)$.
Assume that $y\,{\not=}\,1$.
Then, $K_y$ has defect zero by (i), which means that $p\,{\not |}\,|C_G(y)|$. Since $v$ is a $p$-element
in $C_G(y)^\#$, this is a contradiction. Hence $y=1$, so that $C_E(v)=1$.

(iii)
The second equality holds trivially. So, we look at only the first equality.
Take any element $g\in K_i$ for some $i\in\{1,\cdots,\ell(G)-1\}$. 
Since $G=D\rtimes E$, we can write
$g=vy$ for some $v\in D$ and $y\in E$. Since $g$ is a non-identity $p'$-element, 
$y\,{\not=}\,1$, so that $y\in E^\#$.

Conversely, take any element $g=vy$ with $v\in D$ and $y\in E^\#$. Hence $g\,{\not=}\,1$.
Let $g_p$ and $g_{p'}$ respectively be the $p$-part and the $p'$-part of $g$.
Since $y\,{\not=}\,1$ and since $D\vartriangleleft G$, $g$ is not a $p$-element. Hence $g_{p'}\,{\not=}\,1$.
Thus we can assume that $g_{p'}\in K_1$.
Then, (i) yields that $g_p=1$. That is, $g=g_{p'}$, which means that $g$ is a non-identity $p'$-element,
and hence $g\in\bigsqcup_{i=1}^{\ell(G)-1}\, K_i$.

(iv)
 Let $r$ be the number of all conjugacy classes of $G$ that contain $p$-elements (possibly the identity element).  
Since $k(B)-\ell(B)=2$ by the assumption, it follows by \cite[Theo\-rems 5.4.13(ii) and 5.6.1]{NT89} that
$r=2$ or $3$. If $r=2$, then for every $p$-element $u\in G^\#$, $\ell(B_0(F[C_G(u)]))=2$,
contradicting the assumption that $C_G(u)$ is $p$-nilpotent. Thus, $r=3$. So, we can pick up two
non-identity $p$-elements, say $z, u$ such that $z\in Z(D)$ and that $u$ and $z$ are not conjugate in $G$.
Then, since $D\vartriangleleft G$,
$D=\{1\}\sqcup K_z\sqcup K_u$ (disjoint union).
 
{\sf Suppose that $D$ is non-abelian.} Then, $u$ can be chosen as $u\,{\not\in}\,Z(D)$.
First we claim $K_u\cap Z(D)=\emptyset$.
If there is an element $z' \in K_u\cap Z(D)^{\#}$, then $u$ is $G$-conjugate to an element in $Z(D)$, 
so that $u\in Z(D)$ since $Z(D) \, {\mathrm{char}} \, D\, {\mathrm{char}}\, G$,
a contradiction. Thus, $K_u\cap Z(D)=\emptyset$.

Now, we claim that $Z(D)^\#=K_z$. Since $D\vartriangleleft G$ and 
$Z(D) \, {\mathrm{char}} \, D$,
$C_z\subseteq Z(D)^\#$. Conversely, take any element $z'\in Z(D)^\#$. 
Then, since $K_u\cap Z(D)=\emptyset$ 
from the above, $z' \,{\not\in}\, K_u$, so that $z' \in K_z$. Namely, $Z(D)^\# \subseteq K_z$,
so the claim is proved.
Thus, by (iii)
\begin{equation}\label{decompG}  \ 
G= K_0\sqcup K_z\sqcup K_u\sqcup K_1\sqcup\cdots\sqcup K_{\ell(G)-1}, \ \ 
Z(D)^\#=K_z, \ \ 
D^\#=K_z\sqcup K_u.
\end{equation}

We next claim that $E$ acts on $K_z$ regularly. Clearly $E$ acts on $K_z$ semi-regularly by (ii). Further,
take any two elements $z_1, z_2\in K_z$. As we have seen above, $z_1, z_2\in Z(D)$.
There is an element $g\in G$ such that $z_2=g^{-1}z_1 g$. We can write $g=vy$ for some $v\in D$ and $y\in E$.
Then, $z_2={z_1}^g={z_1}^{vy}={z_1}^y$, so that $z_1$ and $z_2$ are conjugate in $E$.
So that $E$ acts on $K_z$ transitively. Hence $E$ acts on $K_z$ regularly. Therefore 
\begin{equation}\label{K_z} |K_z|=|E|. \end{equation}

Further we claim that $C_G(u)=C_D(u)$. Take any element $g\in C_G(u)$. First note that $C_D(u)$ is a Sylow
$p$-subgroup of $C_G(u)$ since $D\vartriangleleft G$. Clearly, $C_D(u)\vartriangleleft C_G(u)$.
On the other hand, since $C_G(u)$ is $p$-nilpotent by the assumption, we have
$$ C_G(u)=O_{p'}(C_G(u))\times C_D(u).$$
Then, (i) implies that $O_{p'}(C_G(u))=1$, so that $C_G(u)=C_D(u)$.
Thus $|K_u|=|G:C_G(u)|=|G:C_D(u)|=|D||E|/|C_D(u)|=|E||D:C_D(u)|$. Set $e:=|E|$. Hence
\begin{equation}\label{K_u} |K_u|=e|D:C_D(u)|. \end{equation}
Then, 
\begin{align*}
e|D| &= |G| 
\\&= 1+e+e|D:C_D(u)|+|D|(e-1) \text{ \quad by  
(\ref{decompG}), (\ref{K_z}), (\ref{K_u}) and (iii)}
\\&=1+e+e|D:C_D(u)|+|D|e-|D| 
\end{align*}
so that 
\begin{align*} |D|-1 &= e(1+|D:C_D(u)|) 
\\
&= (|Z(D)|-1)(1+|D:C_D(u)|)      \text{ \quad by  (\ref{decompG}), (\ref{K_z})}
\\&
=|Z(D)|+|Z(D)||D:C_D(u)|-|D:C_D(u)|-1
\end{align*}
and hence
\begin{equation}\label{|D|}
 |D|=|Z(D)|+|Z(D)||D:C_D(U)|-|D:C_D(u)|.
\end{equation}
Now, set
$$ |D|=:p^d, \ |Z(D)|=:p^s, \ |C_D(u)|=:p^t  \ \text{ for some integers }d, s, t \geq 1.$$
Thus, by (\ref{|D|}),
$$ p^d = p^s + p^s p^d/p^t - p^d/p^t$$
so that
$$ p^d p^t = p^{s+t} + p^{s+d} - p^d $$
and hence
\begin{equation}\label{p^t} 
            p^t = p^{s+t-d} + p^s - 1.
\end{equation}
Clearly $Z(D) \subsetneqq Z(D)\langle u\rangle \subseteq C_D(u)$ since $u\in D - Z(D)$,
so that $Z(D) \lneqq C_D(u)$, and hence $p^s<p^t$. Hence we can set that
$$  t-s =: n \text{\quad for some integer }n\geq 1. $$
Thus, by (\ref{p^t}),
$$  p^s - 1 = p^t - p^{s+t-d} = p^{s+n} - p^{2s+n-d}. $$
Since $s\geq 1$, $p^s-1$ is a positive integer. Furthermore since $s+n > 1$, we have to have that
$$ 2s+n-d = 0.$$
Hence by the above,
$$ p^s-1 = p^{s+n}-1.$$
So that $s = s+n$, and hence $n=0$, a contradiction.
{\sf Therefore, $D$ is abelian.}

(v) By (iv) and the proof of (iv), we get the assertion immediately. 

(vi) By the assumption, for every $p$-element $u'\in G^\#$,
$C_G(u')$ is $p$-nilpotent, so that $\ell( B_0(F[C_G(u')]))=1$. 
Further, by (iv), every element $u'$ in $D$ is conjugate to an element in $Z(D)$.
Thus, the condition (ii) in \cite[Theorem 3.1]{KL02} is satisfied, so that by the theorem
$\bar Z(B)$ is a symmetric algebra.
On the other hand, by the assumption that $k(B)-\ell(B)=2$ and Lemma \ref{KL02-p.143}, 
we have $\dim_F\bar Z(B)=3$. 
\end{proof}

\noindent
Although the first part of the next lemma is essentially due to Dade,
we give an elementary proof for convenience of the readers.

\begin{Lemma}[Dade, Lemma 9.9 in \cite{Dad99}]
\label{1_B=1_b}
Assume that $N \vartriangleleft G$ with $p\,{\not|}\,|G/N|$, 
and $B$ and $b$ are blocks of $FG$ and $FN$, respectively,
such that $B$ covers $b$. Let $D$ be a defect group of $B$
such that $C_G(D)\leq N$.
\begin{enumerate}
\item[\rm (i)]
If furthermore $b$ is $G$-invariant, then $1_B=1_b$,
namely $B$ is the unique block of $FG$ covering $b$, 
and $b$ is the unique block of $FN$ covered by $B$.
\item[\rm (ii)]
If in particular $B=B_0(FG)$, then,
${\mathrm{Irr}}(B\,|\,{{1_N{\uparrow}}^G})= {\mathrm{Irr}}(G/N)$
where we identify ${\mathrm{Irr}}(G/N)$ as a subset of ${\mathrm{Irr}}(G)$.
\end{enumerate}
\end{Lemma}

\begin{proof}
(i)
Since $b$ is $G$-invariant, it holds that $b$ is the unique block of $FN$ covered by $B$ 
(see \cite[Lemma 5.5.3]{NT89}) and that $1_b\in Z(FG)$. Surely $1_b$ is an idempotent.
Set $B_1:=B$. Hence we can write that
$1_b=1_{B_1} + 1_{B_2}+\cdots +1_{B_n}$ for an integer $n\geq 1$
and for distinct blocks $B_1, \cdots, B_n$ of $FG$ since $B_1$ covers $b$.
Note that $B_1, \cdots, B_n$ are all blocks of $FG$ covering $b$. 
So it suffices to prove that $n=1$.
    
Since $C_G(D)\leq N$ and $D$ is a defect group of $B$, 
\cite[Lemma 5.5.14]{NT89} implies that
$B$ is regular with respect to $N$. Hence, by \cite[Theorem 5.5.13(ii)]{NT89}, $B=b^G$.
Now, by \cite[4.2. Proposition]{Kno76}, there is a defect group $\mathbb D$ of $b$
such that $\mathbb D =_G D\cap N$. Obviously $D\cap N = D$ since
$D/(D\cap N)\cong DN/N \leq G/N$ and $G/N$ is a $p'$-group.
Hence $\mathbb D=_G D$, so that there is an element $g\in G$ such that 
$\mathbb D= g^{-1}Dg$. Then, by replacing $g^{-1}Dg$ by $D$, we can assume that $\mathbb D=D$. 
Namely, $D$ is a defect group of $b$ as well.

Assume that $n\geq 2$. So, there exists the block $B_2$ as above. 
Let $D_2$ be a defect group of $B_2$. 
Since $B_2$ covers $b$ and $D$ is a defect group of $b$,
\cite[4.2. Proposition]{Kno76} implies that $D=_G D_2\cap N$.
Again by replacing some $G$-conjugate of $D_2$ by $D_2$, we can assume that
$D=D_2\cap N$.
Hence $D\leq D_2$, so that $C_G(D_2)\leq C_G(D)$. Thus, the assumption that $C_G(D)\leq N$
yields that $C_G(D_2)\leq N$. Therefore again by \cite[Lemma 5.5.14]{NT89}, $B_2$ is regular
with respect to $N$. 
Hence $B$ and $B_2$ are both regular with respect to $N$ and also both cover $b$,
that contradicts the 
uniqueness in \cite[Theorem 5.5.13(ii)]{NT89}.
 
(ii)
Set $b:=B_0(FN)$.
From (i),  $1_B=1_b$. Since the character $1_N$ 
extends to the character $1_G$, the assertion follows
by \cite[(6.17) Corollary]{Isa76}.
\end{proof}

\begin{Lemma}\label{Gallagher|G/N|=p}
Assume that $N \vartriangleleft G$ such that $G/N$ is a $p$-group, 
and $B$ and $b$ are blocks of $FG$ and $FN$, respectively,
such that $B$ covers $b$. 
\begin{enumerate}
\item[\rm (i)]
If furthermore $b$ is $G$-invariant, then $1_B=1_b$.
\item[\rm (ii)]
If particularly $B=B_0(FG)$, then 
${\mathrm{Irr}}(B\,|\,{{1_N{\uparrow}}^G})= {\mathrm{Irr}}(G/N)$
where we identify ${\mathrm{Irr}}(G/N)$ as a subset of ${\mathrm{Irr}}(G)$.
\end{enumerate}
\end{Lemma}

\begin{proof}
(i) Follows by \cite[Corollary 5.5.6]{NT89}.

(ii) Follows by (i) and \cite[(6.17)Corollary]{Isa76} just as in the proof of Lemma \ref{1_B=1_b}(ii).
\end{proof}

\noindent
The authors thank Radha Kessar for informing them of the next lemma,
which is useful for our aim.	

\begin{Lemma}\label{Radha}
Suppose that {\bf G} is a simple, simply connected linear algebraic group defined over the
algebraic closure $\bar{\mathbb F}_q$ of $\mathbb F_q$ for a power $q$ of a prime,
and that $\mathcal F: \text{\bf G}\rightarrow \text{\bf G}$ is a Steinberg (Frobenius) morphism.
Set $G:=\text{\bf G}^{\mathcal F}$, and let $Z$ be a subgroup of $Z(G)$.
Further, if $\ell$ is a prime with $\ell\geq 5$
and $G/Z$ has an abelian Sylow $\ell$-subgroup,
then we have that $\ell\,{\not|}\,|Z|$ and hence that a Sylow $\ell$-subgroup of $G$
is abelian.
\end{Lemma}

\begin{proof}
It follows from \cite[Table 24.2]{MT11} that if $G$ is neither ${\mathrm{SL}}_n(q)$ for $n\geq 2$
nor ${\mathrm{SU}}_n(q)$ for $n\geq 3$, then $\ell\,{\not|}\,|Z|$ since $\ell\geq 5$.
\\ \noindent
{\bf Case 1:} $G:= {\mathrm{SL}}_n(q)$ for $n\geq 2$:
Suppose $\ell \,\Big|\,|Z|$.
By the above table, $|Z|\,\Big|\, |Z(G)|={\mathrm{gcd}}(n,q-1)$.
Since $\ell\,\Big|\,|Z|\,\Big|\,\text{gcd}(n,q-1)\,\Big |\,(q-1)$, 
there exists  a primitive $\ell$-th root of unity in $(\mathbb F_q)^\times$, which
we denote by $\zeta$.  Let $\mathfrak a$ be an element in $G$ such that
$\mathfrak a:={\mathrm{diag}}(\zeta, \zeta^{-1}, 1, 1, \cdots, 1)$. 
Note that  $\ell \leq n$ since $\ell\,\Big|\,|Z|\,\Big|\,\text{gcd}(n,q-1)\,\Big |\,n$.
Further, let $\mathfrak b$ be an element in ${\mathrm{GL}}_\ell(q)$
such that 
$\mathfrak b$ is a permutation matrix corresponding to the cycle $(1 \, 2\,  3 \cdots \ell)\in\mathfrak S_\ell$, namely, 
$$
\mathfrak b := \begin{pmatrix}
0&1&0&0&\cdots \ 0 \\
0&0&1&0&\cdots \ 0 \\
\vdots &  &\ddots&\ddots  & \ \ \ \ \ \vdots \\          
0&0&0&0&\ddots \ 1\\
1&0&0&0&\ddots \ 0 \end{pmatrix}
\ \ \in {\mathrm{GL}}_\ell (q)
$$
Since $\ell$ is odd, the cycle $(1 \, 2\,  3 \cdots \ell)$ is an even permutation, and hence
$\mathfrak b \in {\mathrm{SL}}_\ell(q)$. 
Then, $\mathfrak b$ is considered as an element of $G={\mathrm{SL}}_n(q)$ 
via ${\mathrm{SL}}_\ell(q) \times \langle I_{n-\ell}\rangle \ \leq \ {\mathrm{SL}}_n(q)$
where $I_{n-\ell}$ is the identity matrix of degree $n-\ell$. It is easy to see that
$\mathfrak a^{-1}\mathfrak b^{-1}\mathfrak a\mathfrak b
= {\mathrm{diag}}(\zeta^{-1}, \zeta^2, \zeta^{-1}, 1, \cdots, 1)$, so that
\begin{equation}\label{commutator} 
         [\mathfrak a, \mathfrak b] :=
          \mathfrak a^{-1}\mathfrak b^{-1}\mathfrak a\mathfrak b \,{\not\in}\, Z
\end{equation}
since $\ell\,{\not=}\,3$.
Now, let $m$ be the largest integer such that $\ell^m\,|\,(q-1)$, so that $m\geq 1$
since $\ell\,|\,(q-1)$ as we have seen above.
Let $Q$ be a subgroup of $G$ such that 
$Q$ is the set of all the diagonal matrices in $G$ 
which satisfy that their diagonal parts are elements in the cyclic group $C_{\ell^m}$
(we consider that $C_{\ell^m} \leq (\mathbb F_q)^\times$).
Clearly, $Q$ is an abelian $\ell$-group and $\mathfrak a\in Q$.
Now, since $\mathfrak b$ acts on $Q$ canonically, we can define
a semi-direct product $P:= Q\rtimes\langle\mathfrak b\rangle$.
Apparently, $P$ is an $\ell$-subgroup of $G$, and $\mathfrak a, \mathfrak b\in P$.
Since we are assuming that a Sylow $\ell$-subgroup of $G/Z$ is abelian,
$P Z/Z$ is abelian, so that $ [\mathfrak a, \mathfrak b]\in Z$,
a contradiction by  (\ref{commutator}). 
\\ \noindent
{\bf Case 2:} $G:= {\mathrm{SU}}_n(q)$ for $n\geq 3$:
Suppose $\ell \,\Big|\,|Z|$. By \cite[Table 24.2]{MT11},
$$\ell\,\Big|\,| Z|\,\Big|\,|Z(G)|={\mathrm{gcd}}(n, q+1)\,\Big |\,(q+1)
   \,\Big|\,(q^2-1)=|(\mathbb F_{q^2})^\times|,$$
so that there is a primitive $\ell$-th root of unity in $(\mathbb F_{q^2})^\times$,
which we denote by $\zeta$
(recall that 
in the classical notation $G$ is denoted by ${\mathrm{SU}}(n, q^2)$, namely,
$G$ is defined over the field $\mathbb F_{q^2})$. 
Note also that $\ell\leq n$ since $\ell\,|\,n$ from the above.
Then, just exactly as in {\bf Case 1} we can define $\mathfrak a$ and $\mathfrak b$, and
precisely by the same argument, and hence
we finally have a contradiction. 
\end{proof}

\noindent
The following wonderful result due to Brou\'e and Michel 
shall play a very important role to prove Theorems \ref{MainTheorem},
\ref{k=4ell=3MainTheorem} and \ref{k=4ell=2MainTheorem}.

\begin{Lemma}[Brou\'e-Michel \cite{BM93}]\label{BroueMichel}
Suppose that {\bf G} is a connected reductive algebraic group defined over the algebraic closure
$\bar{\mathbb F}_q$ of $\mathbb F_q$ for a power $q$ of a prime and 
$\mathcal F: \text{\bf G}\rightarrow\text{\bf G}$ is 
a Steinberg (Frobenius) morphism.
Set $G:= \text{\bf G}^{\mathcal F}$. If furthermore 
a prime $\ell$ satisfies that
$5\leq\ell\,{\not|}\,q$ and a Sylow $\ell$-subgroup 
of $G/Z(G)$ is abelian, 
then the principal $\ell$-blocks of $G$ and $N_{G}(D)$ are
isotypic, where $D$ is a Sylow $\ell$-subgroup of $G$. 
\end{Lemma} 

\begin{proof} 
By Lemma \ref{Radha}, $\ell\,{\not|}\,|Z(G)|$.
Hence, by the assumption, 
a Sylow $\ell$-subgroup of $G$ is abelian. 
Thus, 
\cite[Th\'eor\`eme]{BM93} yields 
the assertion.
\end{proof}

\begin{Lemma}\label{sporadic}
Assume that $G$ be one of the 26 sporadic simple groups,
$D\in{\mathrm{Syl}}_p(G)$, and $B:=B_0(FG)$.
If $p\geq 5$ and $D$ is non-cyclic elementary abelian,
then $k(B)>9$. (Surely this bound is very loose.)
\end{Lemma}

\begin{proof}
By \cite{Atlas} and \cite{ModularAtlas} it is enough to check the following four cases for $G$:
\newline
\ \ {\bf  (1)}{\sf Th} and $|D|=7^2$,  {\bf (2)} ${{\sf Fi}_{24}}'$ and $|D|=5^2$,
 {\bf (3)} {\sf B} and $|D|=7^2$,  {\bf (4)} {\sf M} and $|D|=11^2$
\newline
where the four groups are the Thompson simple group, the derived subgroup of 
Fischer's 24 group, the baby monster and the monster, respectively. 
So  we get the assertion for these by 
{\sf GAP} for the {\sf Th} for $p=7$,
\cite{ACOU08}, \cite{AW04}, \cite{AW10},
respectively.
\end{proof}

\begin{Lemma}\label{the16simples}
Suppose that $G$ is one of the 16 non-abelian simple groups
in the list \cite[Table 24.3]{MT11}  (cf. \cite[Table 6.1.3]{GLS98}).
Further assume that $p\geq 5$ and $D\in{\mathrm{Syl}}_p(G)$ is non-cyclic 
elementary abelian, and let $B:=B_0(FG)$.
Then $k(B)>9$. (Surely this bound is very loose.)
\end{Lemma}

\begin{proof}
By \cite{Atlas} and \cite{ModularAtlas}, we must check only the case where
$G=G(q)= {^2}{\!}E_6(2)$.
By \cite[p.191]{Atlas}, we have to take care of only the cases that
$|D|=5^2$ and $|D|=7^2$.

First assume $p=5$.
By \cite{BMM93, CE99}, 
we have $\ell(B)=16$
since the $e$, that is the order of $q=2$ modulo $\ell=p=5$,
is $4$ and the principal $\Phi_4$-block has 16 unipotent characters
(see \cite[Theorem 3.2 and Table 3]{BMM93}).
Hence we can exclude this case.

Next, assume $p=7$. 
Then, by \cite[Theorem 3.10]{Bre94}, $\ell(B)\geq 9$ (note that the $e$ is three since
$q=2$ and $\ell=p=7$, and note also that by \cite{BMM93, CE99} we can know the precise value of $\ell(B)$
since the $e$ is $3$ and the principal $\Phi_3$-block has $15$ unipotent characters,
though we do not need it). Thus we can exclude this case, too.
\end{proof}

\section{What we can say if $k(B)=3$ and $\ell(B)=2$}

The purpose of this section is to give another proof of 
Belonogov's Theorem. 

\begin{Theorem}[Belonogov \cite{Bel90}]\label{MainTheorem}
If the principal $p$-block $B:=B_0(FG)$ of $FG$ satisfies that $k(B)=3$, 
then a Sylow $p$-subgroup $D$ of $G$ is of order $3$.
\end{Theorem}

\noindent
The result here can be considered as an easy but nice application
of the {\sf CFSG}. Actually even better is that 
by the same method in principal we shall
be able to determine a Sylow $p$-subgroup of $G$ when the number
$k(B)$ is four for the principal $p$-block $B$ of $G$.

\begin{Hypothesis}\label{KeyAss}
We assume the following from now on 
till the end of this section.
\begin{enumerate}
\item[$\bullet$] $D$ is a Sylow $p$-subgroup of a finite group $G$ with $|D|=:p^d$.
\item[$\bullet$] 
Set $B:=B_0(FG)$ with
$k(B)=3$ and $\ell(B)=2$, and set
${\mathrm{Irr}}(B)=:\{ {\chi_0:=1_G}, \, \chi_1, \, \chi_2\, \}$ and
${\mathrm{IBr}}(B)=:\{{\phi_0:=(1_G)^0}, \phi_1 \}$. 
\end{enumerate}
\end{Hypothesis}

\begin{Lemma}\label{k=3p=2Dcyclic}
We have that $p\,{\not=}\,2$,
and that Theorem \ref{MainTheorem} holds provided
$D$ is cyclic or $D\vartriangleleft G$.
\end{Lemma}

\begin{proof}
If $p=2$, then we have a contradiction
by \cite{Dad66} and \cite[Corollary 1.3(i)]{Lan81}.
If $D$ is cyclic, then the assertion follows from \cite{Dad66}.
If $D\vartriangleleft G$, then the assertion is implied by \cite[Theorem 4.1]{KNST14}
(note that this is a {\sf CFSG}-free result though depends on {\sf GAP}).
\end{proof}

\begin{Lemma}\label{ActionOfE}
$D$ is elementary abelian. 
\end{Lemma}

\begin{proof}
Follows by \cite[Theorem 3.6]{KNST14}
(note that this result needs the {\sf CFSG} since this depends on the
\cite[Theorem A]{KNST14}).
\end{proof}

\begin{Lemma}[Navarro-Sambale-Tiep \cite{NST18}]\label{p=3} 
If 
$p=3$, then $|D|=3$.
\end{Lemma}
 
\begin{proof}
First, $k_0(B)\equiv 0$ (mod $3$) by \cite[Corollary 1.6]{Lan81}, so that $k_0(B)=3$. 
Hence \cite[Theorem C]{NST18} implies that $|D|=3$
(note that the {\sf CFSG} is used in \cite{NST18} while never is in \cite{Lan81}
which was before the {\sf CFSG}).
\end{proof}

\noindent
The next aim  
is that we can assume
that $G$ is non-abelian simple to prove Theorem \ref{MainTheorem}.

\begin{Hypothesis}\label{k=3}
Besides Hypothesis \ref{KeyAss}, we assume the following from now on 
till just before the proof of Theorem \ref{MainTheorem} because of Lemmas 
\ref{k=3p=2Dcyclic}, \ref{ActionOfE} and \ref{p=3}.
\begin{enumerate}
\item[$\bullet$] 
$p$ is a prime with $p\geq 5$.
\item[$\bullet$] $D$ is not normal and is a non-cyclic elementary abelian Sylow $p$-subgroup of $G$ with $|D|=:p^d$.
\item[$\bullet$]
For every finite group $\mathfrak G$ such that $|\mathfrak G|<|G|$
it never happens that
$k(B_0(F\mathfrak G))=3$.
\end{enumerate}
\end{Hypothesis}

\begin{Lemma}\label{Op'Op}
We can assume that $O_{p'}(G)=1$ and $O_p(G)=1$.
\end{Lemma}

\begin{proof}
Since $B$ is principal, we can assume that $O_{p'}(G)=1$
(see \cite[Theorem 5.8.1]{NT89}).
Assume that $Q:=O_p(G)\,{\not=}\,1$. Set $\bar G:=G/Q$, $\bar D:=D/Q$ 
and $\bar B:=B_0(F\bar G)$.
Then, $B$ dominates $\bar B$ since both blocks contain the trivial characters.
So, by \cite[Theorem 5.8.10 and Lemma 5.8.6(ii)]{NT89}, 
${\mathrm{Irr}}(\bar B)\subseteq{\mathrm{Irr}}(B)$,
and hence $k(\bar B)=1, 2$ or $3$ by Hypothesis \ref{KeyAss}.
    
If $k(\bar B)=1$, then $\bar D=1$, so that $D\vartriangleleft G$, a contradiction by Hypothesis \ref{k=3}.
    
Assume $k(\bar B) = 2$. Then $\ell(\bar B)=1$, 
so that $\bar G$ is $p$-nilpotent (see \cite[Theorem 5.8.3]{NT89}), and hence $G$ is $p$-solvable. 
Since $D$ is abelian by Lemma \ref{ActionOfE}, 
$G$ is $p$-solvable of $p$-length one by \cite[Theorem 6.3.3]{Gor68}.
So $|D|=3$ by Lemma \ref{k=3p=2Dcyclic}.
This is a contradiction as above.
    
Finally, if $k(\bar B)=3$, then ${\mathrm{Irr}}(B)={\mathrm{Irr}}(\bar B)$,
so that $Q\leq \bigcap_{\chi\in{\mathrm{Irr}}(B)}\,\ker(\chi)$,
and hence $Q \leq O_{p'}(G)$ by \cite[Theorem 5.8.1]{NT89}, a contradiction.
\end{proof}

\noindent
Actually, the following is the key lemma of this section, 
which makes it possible for us to reduce to the case where
$G$ is non-abelian simple.

\begin{Lemma}\label{O^p'(G)}
We can assume that $O^{p'}(G)=G$.
\end{Lemma}

\begin{proof}
Assume that $O^{p'}(G)\,{\not=}\,G$. Then, there is a normal subgroup $N\vartriangleleft G$
such that  $p\,{\not|}\,|G/N|$ and $G/N$ is a simple group. 
Set $b:=B_0(FN)$, and $H:=N{\cdot}C_G(D)$.
Then, $H\vartriangleleft G$ by the Frattini argument.
Since $G/N$ is simple,
$H = G$ or $H=N$. 

Assume first that $H = G$. 
Then, by \cite{Alp76} and \cite[Theorem]{Dad77}, $B\cong b$ as $F$-algebras.
Hence $k(b)=3$ (note that $\dim_F\,Z(B)=k(B)$).
But this is a contradiction by Hypothesis \ref{k=3}.
    
So we have $H=N$. 
Thus, by Lemma \ref{1_B=1_b}(ii), ${\mathrm{Irr}}(B\,|\,{{1_N{\uparrow}}^G})= {\mathrm{Irr}}(G/N)$. 
So that $k(G/N)\leq 3$, which implies that  $G/N$ is solvable by \cite[Table 1, p.309]{LL85}.
Since $G/N$ is simple, $G/N \cong C_r$ for a prime $r$ with $r\,{\not=}\,p$.
Since $G/N$ is cyclic and $1_N$ is $G$-invariant,
Clifford's Theorem and \cite[Chapter 3 Problem 11(i)]{NT89} imply that
\begin{equation}\label{InnerProduct} 
\text{if }\chi\in{\mathrm{Irr}}(G) \text{ with }\langle \chi{\downarrow}_N, 1_N\rangle^N\,{\not=}\,0, 
\text{ then }\chi{\downarrow}_N = 1_N, \text{ so }\chi(1)=1.
\end{equation} 
Hence, it follows by Lemma \ref{1_B=1_b}(i) and the Frobenius Reciprocity 
(and also by recalling that ${\mathrm{Irr}}(B)=\{\chi_0, \chi_1, \chi_2\}$ in Hypothesis \ref{KeyAss})
that
$$
1_N{\uparrow}^G \in\{ \chi_0:=1_G,\, \chi_0+\chi_1, \, \chi_0+\chi_2,\, \chi_0+\chi_1+\chi_2  \}.
$$
Clearly, $1_N{\uparrow}^G \not= \chi_0$ since $|G/N|=r>1$.
So $\langle 1_N{\uparrow}^G, \chi_1\rangle^G \not= 0$ or
$\langle 1_N{\uparrow}^G, \chi_2\rangle^G \not= 0$. So, without loss of generality,
we can assume the first case.
Then, $\chi_1{\downarrow}_N=1_N$ by (\ref{InnerProduct}), so that $\chi_1(u)=1$ for every
$p$-element $u\in G$ since $p {\not|}\, |G/N|$ (and hence $u\in N)$.
Since $\chi_1 \not= \chi_0=1_G$, there exists an element $g\in G$ with $\chi_1(g)\not= 1$.
Write $g=g_p\,g_{p'}=g_{p'}\,g_p$ where $g_p$ and $g_{p'}$
respectively are the $p$-part and the $p'$-part of $g$. 
Thus, since $\chi_1$ is linear (and hence $\chi_1$ is a group homomorphism)
and since $\chi_1(g_p)=1$, we have that
$$
1 \not= \chi_1(g)=\chi_1(g_p\,g_{p'})=\chi_1(g_p)\,\chi_1(g_{p'})=\chi_1(g_{p'}).
$$
This means that $(\chi_1)^0 \not= \phi_0$ (recall that $\phi_0=(1_N)^0)$.
Then, since ${\mathrm{IBr}}(B)=\{\phi_0, \phi_1\}$ by Hypothesis \ref{KeyAss},
and since $\chi_1(1)=1$, we must have that $(\chi_1)^0=\phi_1$. This yields that
$\phi_0(1)=\phi_1(1)=1$. Therefore, by \cite[Theorem $(7)\Leftrightarrow (10)$]{Kos81, Kos90},
$B\cong\mathbb B$ as $F$-algebras where $\mathbb B :=B_0(F\,N_G(D))$.
Hence, by Hypothesis \ref{KeyAss} and \cite[Lemma 5.11.3]{NT89}, 
we have $3=k(B)=k(\mathbb B)$.
This is a contradiction by Hypothesis \ref{k=3}. 
\end{proof}

\begin{Lemma}\label{reduction}
We can assume that $G$ is a non-abelian simple group.
\end{Lemma}

\begin{proof}
By Lemmas \ref{Op'Op} and \ref{O^p'(G)}, we have that $O_{p'}(G)=1=O_p(G)$ and $G=O^{p'}(G)$.
Since 
$D$ 
is abelian by Lemma \ref{ActionOfE}, it follows from
\cite[2.1.Theorem]{KS95} 
that there are non-abelian simple groups $G_1, \cdots, G_n$ such that 
$p\,\Big|\,|G_i|$ for each $i$ and that $G = G_1 \times \cdots \times G_n$ 
for some $n\in\mathbb N$.
Set $B_i:= B_0(FG_i)$ for each $i$.
Then $B \cong B_1\otimes_F \cdots\otimes_F B_n$ as $F$-algebras, 
so that $\ell(B) = \ell(B_1) \times \cdots \times\ell(B_n)$. Since each $G_i$ is not $p$-nilpotent, 
we have that $\ell(B_i)\geq 2$ for every $i$. Since $\ell(B)=2$ by Hypothesis \ref{KeyAss}, 
it has to hold that $n=1$.
\end{proof}

\begin{Lemma}\label{NonDefChar}
If we assume moreover that $G$ is a finite simple group $G(q)$ of Lie type defined over $\mathbb F_q$
such that $p\,{\not|}\,q$, then we get a contradiction.
\end{Lemma}

\begin{proof}
First of all, we can assume that $G\,{\not\cong}\, ^2\!{F_4(2)}'$
since if $G={^2}\!{F_4(2)}'$, then 
$k(B)\geq 8$ for all primes $p\,|\,|G|$ by \cite[p.74]{Atlas} and
\cite{ModularAtlas}, that contradicts Hypothesis \ref{KeyAss}.
By Hypothesis \ref{k=3}, $p\geq 5$ and $D$ is non-cyclic elementary abelian.
Thus Lemma \ref{the16simples} yields that
there are a simple, and simply connected reductive algebraic group {\bf G} 
defined over the algebraically closed field $\overline{\mathbb F}_q$ 
and a Steinberg (Frobenius) endomorphism
$\mathcal F: \text{\bf G} \rightarrow \text{\bf G}$ with
$G = {\text{\bf G}}^{\mathcal F}/Z({\text{\bf G}}^{\mathcal F})$
(see \cite[2.3]{Mal17} or \cite[p.185]{Cab18}). 
It follows from
Lemma \ref{BroueMichel} that 
$k(B)=k(\mathbb B)$ where $\mathbb B:=B_0(F\,N_G(D))$.
So that by Hypothesis \ref{KeyAss}, $k(\mathbb B)=3$.
This contradicts Hypothesis \ref{k=3}.
\end{proof}

\noindent
Now, we are ready to reprove Belonogov's Theorem.

\begin{proof}[{\bf Proof of Theorem \ref{MainTheorem}.}]
We use induction on $|G|$. 
If $D\vartriangleleft G$, then $|D|=3$ by Lemma \ref{k=3p=2Dcyclic}.
If $p=2$ or $D$ is cyclic, then $|D|=3$ by
Lemma \ref{k=3p=2Dcyclic}. 
Furthermore if  $p=3$ then $|D|=3$ by Lemma \ref{p=3}.
Thus, we can assume the same as Hypothesis \ref{k=3}. 
So, we can assume that $G$ is non-abelian simple by Lemma \ref{reduction}.

First, assume that $G$ is an alternating group $\mathfrak A_n$ for some $n\geq 5$.
Then, 
\cite[Theorem 5.4]{MO91}, Lemma \ref{ActionOfE} 
and \cite[Consequence 6]{Alp87} imply that 
$3=k(B)=k(\mathbb B)$, where $\mathbb B:=B_0(F\,N_G(D))$
, a contradiction.
 
If $G$ is one of the 26 sporadic simple groups,
then the assertion follows by Lemma \ref{sporadic}.
    
Assume that $G=G(q)$ is a simple group of Lie type defined over  
$\mathbb F_q$ with $p\,|\,q$. Then $G={\mathrm{PSL}}_2(q)$
since $D$ is abelian
(see e.g. \cite[Theorem]{SZ16}). Thus, by \cite[p.588]{BN41}, \cite[Satz 9.1]{Bur76} and
\cite[8.5. Theorem p.70]{Hum06}, we have that
$\ell(B)=(q-1)/2$, and hence $q=5$ since $\ell(B)=2$, so that $|D|=5$, a 
contradiction since $D$ is non-cyclic.
    
Finally suppose that $G=G(q)$ is a simple group of Lie type 
defined over $\mathbb F_q$ for a power $q$ of a prime such that $p\,{\not|}\,q$.
Then, Lemma \ref{NonDefChar} yields a contradiction.

Therefore, by the classification of finite simple groups {\sf CFSG} in \cite[p.6]{GLS94},
we get a contradiction. The proof is completed.
\end{proof}

\begin{Remark}\label{Sambale}
There is an interesting observation by Sambale in 
\cite[Proposition 15.7]{Sam14}, that says even without the {\sf CFSG} 
that e.g. $p\geq 11$, $d$ is odd if $k(B)=3$ and $\ell(B)=2$. 
\end{Remark}

\section{What we can say if $k(B)=5$ and $\ell(B)=3$}

The purpose of this section is to state lemmas which shall be useful for 
the proof of Theorem \ref{k=4ell=3MainTheorem}.

\begin{Lemma}\label{k=5ell=3}
Let $D\in{\mathrm{Syl}}_p(G)$ such that
$D\vartriangleleft G$ and $D$ is non-cyclic abelian.
Let $B:=B_0(FG)$.
Then, it does not happen that $k(B)=5$ and $\ell(B)=3$.
\end{Lemma}

\begin{proof}
Assume that such a case occurs.
Since $B$ is the principal $p$-block, we can assume that $O_{p'}(G)=1$.
We can write $G=D\rtimes E$ for a $p'$-subgroup $E$ by Schur-Zassenhaus' Theorem.
Clearly $G$ is $p$-solvable, so that Fong-Gasch{\"u}tz's Theorem says that
$G$ has only one $p$-block, say $B$. Hence
$k(G)=k(B)=5$ and $3=\ell(B)=\ell(G)=\ell(G/D)=\ell(E)=k(E)$. Thus, by \cite[Table 1]{LL85}, we know that 
$G\in\{  C_5, D_8, Q_8, D_{14}, C_5\rtimes C_4, C_7\rtimes C_3, \mathfrak S_4, \mathfrak A_5\}$
and  $E\in\{ C_3, \mathfrak S_3 \}$. So that $p \not= 3$.
Since $3 \Big| |E|\Big| |G|$, $G$ is $p$-solvable and $D$ is non-cyclic, we have 
$G=\mathfrak S_4$. So that $p=2$, and this is a contradiction since 
$D$ is abelian. 
\end{proof}

\begin{Lemma}\label{k=5ell=3SimpleGp}
Suppose that $B:=B_0(FG)$ with $k(B)=5$ and $\ell(B)=3$.
Assume, further, that $p\geq 5$ and a Sylow $p$-subgroup of $G$ 
is non-cyclic elementary abelian.
\begin{enumerate}
\item[(i)]
If $G$ is a non-abelian simple group, then this case does not happen.
\item[(ii)]
If $G=O^{p}(G)=O^{p'}(G)$ and $O_{p'}(G)=O_p(G)=1$, then this case does
not happen.
\end{enumerate}
\end{Lemma}

\begin{proof}
(i) Assume that such a case happens. We use the {\sf CFSG}.
Let $D\in{\mathrm{Syl}}_p(G)$.

First, assume that $G$ is an alternating group $\mathfrak A_n$ for some $n\geq 5$.
Then, 
since $D$ is abelian
(and hence by Burnside's Theorem), \cite[Theorem 5.4]{MO91} and 
\cite[Consequences 5 and 6]{Alp87} imply that 
$5=k(B)=k(\mathbb B)$ and $3=\ell(B)=\ell(\mathbb B)$, 
where $\mathbb B:=B_0(F\,N_G(D))$. 
Thus we have a contradiction by Lemma \ref{k=5ell=3}.

If $G$ is one of the 26 sporadic simple groups, then the assertion follows by Lemma \ref{sporadic}.

Assume that $G=G(q)$ is a simple group of Lie type defined over  
$\mathbb F_q$ with $p\,|\,q$. Then $G={\mathrm{PSL}}_2(q)$
since $D$ is abelian
(see e.g. \cite[Theorem]{SZ16}). Thus, by \cite[p.588]{BN41}, \cite[Satz 9.1]{Bur76} and
\cite[8.5. Theorem p.70]{Hum06}, we have that
$\ell(B)=(q-1)/2$, and hence $q=7$ since $\ell(B)=3$, so that 
$|D|=7$ since $p\,|\,q$, a contradiction since $D$ is non-cyclic.

Finally suppose that $G=G(q)$ is a simple group of Lie type defined
over a certain finite field $\mathbb F_q$ for some $q$ with $p\,{\not|}\,q$.
Since $p\geq 5$, $D$ is non-cyclic elementary abelian and $k(B)=5$, 
all the sixteen finite simple groups in \cite[Table 24.3]{MT11}  
(cf. \cite[Table 6.1.3]{GLS98}) cannot be our $G$
by Lemma \ref{the16simples}.

As we have seen in the proof of Lemma \ref{NonDefChar},
we can assume that $G\,{\not\cong}\, ^2\!{F_4(2)}'$
since if $G={^2}\!{F_4(2)}'$, then $k(B)\geq 8$.
Then, just as in the proof of Theorem \ref{MainTheorem} we can assume that 
there are a simple, and simply connected reductive 
algebraic group {\bf G} 
defined over the algebraically closed field $\overline{\mathbb F}_q$ 
and a Steinberg (Frobenius) endomorphism
$\mathcal F: \text{\bf G} \rightarrow \text{\bf G}$ with
$G = {\text{\bf G}}^{\mathcal F}/Z({\text{\bf G}}^{\mathcal F})$
(see \cite[2.3]{Mal17} or \cite[p.185]{Cab18}). 
Then we get by
Lemma \ref{BroueMichel} that 
$k(B)=k(\mathbb B)$ and $\ell(B)=\ell(\mathbb B)$ 
where $\mathbb B:=B_0(F\,N_G(D))$.
Hence $k(\mathbb B)=5$ and $\ell(\mathbb B)=3$, a
contradiction by Lemma \ref{k=5ell=3}.

Therefore just as in the proof of Theorem \ref{MainTheorem},
by making use of the {\sf CFSG} in \cite[p.6]{GLS94}, the proof of (i) is completed.

(ii) It follows from \cite[2.1.Theorem]{KS95} 
that there are non-abelian simple groups $G_1, \cdots, G_n$ such that 
$p\,\Big|\,|G_i|$ for each $i$ and that $G = G_1 \times \cdots \times G_n$ 
for some $n\in\mathbb N$.
Set $B_i:= B_0(FG_i)$ for each $i$.
Then $B \cong B_1\otimes_F \cdots\otimes_F B_n$ as $F$-algebras, 
so that $\ell(B) = \ell(B_1) \times \cdots \times\ell(B_n)$. Since each $G_i$ is not $p$-nilpotent, $\ell(B_i)\geq 2$ for all $i$, so that $n=1$ since $\ell(B)=3$.
Hence the assertion follows by (i).
\end{proof}

\section{What we can say if $k(B)=4$ and $\ell(B)=3$}

The purpose of this section is to prove the following theorem:

\begin{Theorem}\label{k=4ell=3MainTheorem}
Let $B:=B_0(FG)$ and $D\in{\mathrm{Syl}}_p(G)$. 
If $k(B)=4$ and $\ell(B)=3$, then $D\cong C_2\times C_2$.
\end{Theorem}

\noindent
In the following we list many lemmas whose aim is just to give a complete proof of
Theorem \ref{k=4ell=3MainTheorem}. Because of this we shall use the following notation and 
assume the following entirely throughout this section from now on. Namely,

\begin{Notation}\label{H}
Throughout this section we assume that $G$ is a finite group with a Sylow
$p$-subgroup $D$ with $|D|=:p^d$ for some $d\geq 1$ and 
$B:=B_0(FG)$.
\end{Notation}

\begin{Lemma}\label{nonCyclicElemAbe}
If $k(B)=4$ and $\ell(B)=3$, then we have that 
$D$ is non-cyclic and elementary abelian and that $p\,{\not=}\,3$.
\end{Lemma}

\begin{proof} 
Assume first that $D$ is cyclic. By \cite{Dad66}, the inertial index $e$ of $B$ is 3
since $e=\ell(B)$. Then, $|D|=me+1$, where $m$ is the multiplicity of the 
exceptional characters of $B$, and also $k(B)=e+m$. Hence $4=3+m$, so that
$m=1$. Thus, $|D|=3+1=4$, and hence $p=2$. This is a contradiction since
$e|(p-1)$.

Since $k(B)-\ell(B)=1$ and $B$ is principal,
$D$ is elementary abelian by \cite[Theorem 3.6]{KNST14}
(note that this result needs the {\sf CFSG} since it depends on 
\cite[Theorem A]{KNST14}).

Finally suppose $p=3$.
Since $D$ is abelian, by \cite{KM13}, $k(B)=k_0(B)$.
By \cite[Corollary~1.6]{Lan81}, $k_0(B)\equiv 0$ (mod $3$)
since $D$ is non-cyclic.
Since $k(B)=4$, this is a contradiction.
\end{proof}

\begin{Lemma}\label{k=4ell=3Normal_D}
If $k(B)=4$, $\ell(B)=3$ and $D\vartriangleleft G$, then
$D\cong C_2\times C_2$, in fact $B\cong F\,\mathfrak A_4$
as $F$-algebras, and also as interior $D$-algebras.
\end{Lemma}

\begin{proof}
Since $B$ is principal, we can assume that $O_{p'}(G)=1$. Then 
by Fong-Gasch{\"u}tz's
Theorem, $FG$ has only one block, say $B$. Then, by Schur-Zassenhaus' Theorem,
$G=D\rtimes E$ for a $p'$-subgroup $E$ of $G$.
Thus, $k(G)=k(B)=4$ and $3=\ell(B)=\ell(G)=\ell(G/D)=\ell(E)=k(E)$.
Hence, \cite[Table 1]{LL85} implies that
$$ E\in\{C_3, \mathfrak S_3\} \text{ and }G\in\{C_4, C_2\times C_2, \mathfrak A_4\}.
$$
Then, since $3{\Big|}|E|{\Big|}|G|$, $G=\mathfrak A_4$, so that the assertion follows.
\end{proof}

\begin{Hypothesis}\label{Hypo_k=4ell=3}
Suppose that $k(B)=4$ and $\ell(B)=3$ till 
just before the proof of Theorem \ref{k=4ell=3MainTheorem}.
Further, because of Lemmas \ref{nonCyclicElemAbe} and \ref{k=4ell=3Normal_D}, we can assume that
$p\,{\not=}\,3$, $D$ is non-cyclic elementary abelian and also that
Theorem \ref{k=4ell=3MainTheorem} holds for any finite group $\mathfrak G$ such that
$|\mathfrak G|<|G|$.
\end{Hypothesis}

\begin{Lemma}\label{k=4ell=3SimpleGp}
It holds the following:
\begin{enumerate}
\item[(i)] If $G$ is non-abelian simple, then
$D\cong C_2\times C_2$, so that
$G\cong {\mathrm{PSL}}_2(q)$ 
for a power $q$ of a prime such that $q\equiv 3\text{ or }5$ {\rm{(mod 8)}}.
\item[(ii)]
If $G=O^{p}(G)=O^{p'}(G)$ and $O_{p'}(G)=O_p(G)=1$, then
$G$ is non-abelian simple, and hence 
$D\cong C_2\times C_2$ and further
$G\cong {\mathrm{PSL}}_2(q)$ 
for a power $q$ of a prime such that $q\equiv 3\text{ or }5$ {\rm{(mod 8)}}.
\end{enumerate}
\end{Lemma}

\begin{proof}
(i) As in Lemma \ref{k=5ell=3SimpleGp}, we have a contradiction if $p\geq 5$. So
by Hypothesis \ref{Hypo_k=4ell=3}, $p=2$. Then since $D$ is abelian 
by Hypothesis \ref{Hypo_k=4ell=3},
\cite{FH93} and Lemma \ref{k=4ell=3Normal_D} imply the assertion.

(ii)  By using (i) we can prove (ii) just as in the proof of  Lemma \ref{k=5ell=3SimpleGp}(ii).
\end{proof}

\begin{Lemma}\label{Op'(G)Op(G)}
We can assume that $O_{p'}(G)=O_p(G)=1$.
\end{Lemma}

\begin{proof}
Since $B$ is principal, we can assume $O_{p'}(G)=1$.
Next, assume that $O_p(G)\,{\not=}\,1$. Set $Q:=O_p(G)$, $\bar D:=D/Q$,
$\bar G:=G/Q$, and $\bar B:=B_0(F\bar G)$. Since $B$ dominates $\bar B$,
${\mathrm{Irr}}(\bar B) \subseteq {\mathrm{Irr}}(B)$ 
(see \cite[Theorem 5.8.10 and Lemma 5.8.6(ii)]{NT89}).
So $k(\bar B)=1, 2, 3$ or $4$.

If $k(\bar B)=1$, then $p\,{\not\Big|}\,|\bar G|$, so that $Q\in{\mathrm{Syl}}_p(G)$. 
So the assertion follows by Lemma \ref{k=4ell=3Normal_D}.

If $k(\bar B)=4$, then 
$Q\leq\bigcap_{\chi\in{\mathrm{Irr}}(B)}\,\ker\chi=O_{p'}(G)$, a contradiction
since $Q\,{\not=}\,1$.

Assume $k(\bar B)=2$. Then \cite[Theorem A]{Bra82} implies that $|\bar D|=2$,
and hence $p=2$ and $\bar G$ is $2$-nilpotent by \cite[Theorem 7.6.1]{Gor68}.
Hence 
$G$ is $2$-solvable.
Then, $G$ is of $2$-length one since $D$ is abelian by Lemma \ref{nonCyclicElemAbe}.
Further, since we assume that $O_{2'}(G)=1$, $D\vartriangleleft G$.
Thus the assertion follows from Lemma \ref{k=4ell=3Normal_D}.

Assume finally that $k(\bar B)=3$.
Then, Theorem \ref{MainTheorem} yields that $|\bar D|=3$,
so that $p=3$, a contradiction by Lemma \ref{nonCyclicElemAbe}.
\end{proof}

\begin{Lemma}\label{k=4O^p(G)}
We can assume that $O^p(G)=G$.
\end{Lemma}

\begin{proof}
Suppose $O^p(G)\,{\lneqq}\,G$. Then, there is an $N$ with 
$N\vartriangleleft G$ and $|G/N|=p$. 
Since $D$ is elementary abelian and non-cyclic by Lemma \ref{nonCyclicElemAbe}, 
we can write $G=N\rtimes R$ and $D=Q\times R$
where $Q\in{\mathrm{Syl}}_p(N)$ and $R\leq G$
with $R\cong C_p$. Set $b:=B_0(FN)$. Then it follows from \cite[Theorem]{KK96}
that $B\cong b\otimes_F FR$ as $F$-algebras. 
Then, by \cite[Lemma 2.4.2(iii)]{NT89},
$Z(b)\otimes_F FR = Z(b)\otimes_F Z(FR) = Z(b\otimes_F FR)\cong Z(B)$ as $F$-algebras.
Since $\dim_F Z(B)=k(B)=4$ by Hypothesis \ref{Hypo_k=4ell=3}, it holds that
$k(b)\cdot p = 4$.
Hence $k(b)=2$ and $p=2$.
Then, \cite[Theorem A]{Bra82} implies that $Q\cong C_2$.
So that $D\cong C_2\times C_2$. Namely, the conclusion of 
Theorem \ref{k=4ell=3MainTheorem} holds
(recall that we are discussing by induction, see Hypothesis \ref{Hypo_k=4ell=3}).
\end{proof}

\begin{Lemma}\label{AlperinDade}
If furthermore $G=O^{p'}(G){\cdot}C_G(D)$, then 
$D\cong C_2\times C_2$.
\end{Lemma}

\begin{proof}
By Lemmas \ref{Op'(G)Op(G)} and \ref{k=4O^p(G)}      we can assume that
$O_{p'}(G)=O_p(G)=1$ and $O^p(G)=G$.
Set $L:=O^{p'}(G)$ and $b':=B_0(FL)$.
Then, it follows from \cite{Alp76, Dad77} that $B\cong b'$ as $F$-algebras, so that
$k(b')=4$ and $\ell(b')=3$ by \cite[Lemma 5.11.3]{NT89}.
Obviously $O_{p'}(L)=O_p(L)=1$.
Now, since $O^p(G)=G$, \cite[Chap.I Corollary 10.13]{Lan83} implies that ${\mathrm{Ext}}_B^1(F,F)=0$,
so that ${\mathrm{Ext}}_{b'}^1(F,F)=0$ since the isomorphism $B\cong b'$ sends $F_G$ to $F_L$,
which yields that $O^p(L)=L$ again by \cite[Chap.I Corollary 10.13]{Lan83}.
By the definition of $L$, $O^{p'}(L)=L$.
Hence Lemma \ref{k=4ell=3SimpleGp}(ii) yields that $D\cong C_2\times C_2$.
\end{proof}

\noindent
Now, we are ready to prove our main theorem of this section,
that is in fact the first half of the main result of this paper (see Theorem \ref{k=4MainTheorem}).
Namely,

\begin{proof}[{\bf Proof of Theorem \ref{k=4ell=3MainTheorem}}]
By Lemmas \ref{Op'(G)Op(G)} and \ref{k=4O^p(G)}, we can assume that $O_{p'}(G)=O_p(G)=1$ and $O^p(G)=G$.
Set $N:=O^{p'}(G){\cdot}C_G(D)$ and $b:=B_0(FN)$.
Thus, $N\vartriangleleft G$ by the Frattini argument,  and obviously $C_G(D)\leq N$.
If $G=N$, then Lemma \ref{AlperinDade} yields that $D\cong C_2\times C_2$.
Hence we can assume that $N\lneqq G$. Set $\bar G:=G/N$ and $\bar B:=B_0(F\bar G)$.
Now, by Lemma \ref{1_B=1_b}(i)--(ii), we know that 
\begin{equation}\label{Dade99} 
1_B=1_b \text{ \ \  and \ \  }{\mathrm{Irr}}(\bar G)={\mathrm{Irr}}(B\,|\, {1_N}{\uparrow}^G) \subseteq {\mathrm{Irr}}(B).
\end{equation}
Since $k(B)=4$ and $|\bar G|>1$, $k(\bar G)=2, 3 \text{ or }4$.

Assume first that $k(\bar G)=4$. Then, by \cite[Table 1 p.309]{LL85}, 
$\bar G\in\{C_4, C_2\times C_2, D_{10}, \mathfrak A_4\}$.
If $\bar G\cong C_4\text{ or }C_2\times C_2$, 
then all irreducible ordinary characters in $B$ are linear, so that
\cite[Theorem (7)$\Leftrightarrow$(10)]{Kos81, Kos90}, 
that $k(B)=k(\mathbb B)$ and $\ell(B)=\ell(\mathbb B)$ where $\mathbb B:=B_0(F\,N_G(D))$,
so that Lemma \ref{k=4ell=3Normal_D} yields that $D\cong C_2\times C_2$.
Suppose next that $\bar G\cong D_{10} \text{ or }\mathfrak A_4$. Then $|\{ \chi(1)\,|\, \chi\in{\mathrm{Irr}}(B)\}|=2$. 
Apparently, $p\,{\not=}\,2$ since $p\,{\not|}\,|\bar G|$. Hence by Lemma \ref{nonCyclicElemAbe}, $p\geq 5$.
Thus, by \cite[Theorem A]{GRSS19}, $G$ is $p$-solvable. 
Hence, by Lemma \ref{nonCyclicElemAbe} and \cite[Theorem 6.3.3]{Gor68},
$G$ is of $p$-length one. 
So $D\vartriangleleft G$ since $O_{p'}(G)=1$.
Thus, Lemma \ref{k=4ell=3Normal_D} yields the assertion. 

Assume next that $k(\bar G)=3$. Then, again by \cite[Table 1 p.309]{LL85}, $\bar G\cong C_3\text{ or }\mathfrak S_3$.
If $\bar G\cong C_3$, then $B$ has three distinct linear ordinary characters $\chi_0$, $\chi_1$ and $\chi_2$,
so that $(\chi_i)^0$ for $i=1,2,3$ are three different linear Brauer characters in $B$, and hence
again by  \cite[Theorem (7)$\Leftrightarrow$(10)]{Kos81, Kos90} and Lemma \ref{k=4ell=3Normal_D},
we have $D\cong C_2\times C_2$. Now, assume that $\bar G\cong\mathfrak S_3$.
Let $\chi$ be the unique ordinary character $\mathfrak S_3$ of degree $2$, and consider $\chi$ as
an ordinary character in ${\mathrm{Irr}}(B)$ (see (\ref{Dade99})).
Then, (\ref{Dade99}) and the Frobenius Reciprocity imply that 
$\langle \chi{\downarrow}_N, 1_N\rangle^N\,{\not=}\,0$. But since 
${\mathrm{H}}^2(\mathfrak S_3, \mathbb C^\times)=1$
(see \cite[2.12.3 Theorem (i)]{Kar87}) and since $1_N$ is $G$-invariant, we get by Clifford's Theorem 
(see \cite[Chap.3, the solution of Problem 11]{NT89}) that
$\chi{\downarrow}_N=1_G$, a contradiction since $\chi(1)=2$.

Assume finally that $k(\bar G)=2$. Hence $\bar G\cong C_2$ by \cite[Table 1]{LL85}.
By (\ref{Dade99}), ${\mathrm{Irr}}(B)$ has at least two linear characters,
say $\chi_0:=1_G$ and $\chi_1$ (sign character).  
Since $\bar G$ is cyclic and $1_N$ is $G$-invariant,
by Clifford's Theorem and (\ref{Dade99}),
$\chi_1{\downarrow}_N=1_N$, so that $\chi_1(u)=1$ for every
$p$-element $u\in G$ since $p {\not|}\, |G/N|$ (and hence $u\in N)$.
Since $\chi_1 \not= \chi_0=1_G$, there exists an element $g\in G$ with $\chi_1(g)\not= 1$.
Write $g=g_p\,g_{p'}=g_{p'}\,g_p$ where $g_p$ and $g_{p'}$
respectively are the $p$-part and the $p'$-part of $g$. 
Thus, since $\chi_1$ is linear (and hence $\chi_1$ is a group homomorphism)
and since $\chi_1(g_p)=1$, we have that
$$
1 \not= \chi_1(g)=\chi_1(g_p\,g_{p'})=\chi_1(g_p)\,\chi_1(g_{p'})=\chi_1(g_{p'}).
$$
This means that $(\chi_1)^0 \not= \phi_0$ (recall that $\phi_0=(1_N)^0)$.
Thus we can set $\phi_1:=(\chi_1)^0$. Since $\ell(B)=3$, there is another one, say $\phi_2$
so that ${\mathrm{IBr}}(B)=\{\phi_0, \phi_1, \phi_2\}$ and $\phi_0:=(1_N)^0$.
Let $S_i$ be a simple $FG$-module in $B$ with $S_i\leftrightarrow\phi_i$ for $i=0,1,2$.
Obviously, $\dim_F\,S_i=1$ and $S_i{\downarrow}_N \cong F_N$ for $i=0,1$. Since $1_B=1_b$ by (\ref{Dade99}),
there is a simple $FN$-module $T$ in $b$ such that 
${\mathrm{Hom}}_{FG}(T{\uparrow}^G, S_2)\,{\not=}\,0$. Then, again by Clifford's Theorem
and by the fact that $\bar G\cong C_2$, 
$$ S_2{\downarrow}_N \cong T \text{ \ \ or \ \ } T\oplus T^g$$
where $g\in G$ with $G=\langle N,g\rangle$.

Assume that $S_2{\downarrow}_N \cong T$.
Since $1_B=1_b$ by (\ref{Dade99}) and since $S_i{\downarrow}_N \cong F_N$ for $i=0,1$,
we get that ${\mathrm{IBr}}(b)=\{\psi_0, \psi_1 \}$ where $\psi_0:=(1_N)^0$ and $\psi_1\leftrightarrow T$.
Recall that $|\bar G|=2$ and this is a $p'$-number. Hence, by Maschke's Theorem, $T{\uparrow}^G$ is semi-simple
(recall that $J(FG)=FG{\cdot}J(FN)=J(FN){\cdot}FG)$.
Now, by the Frobenius Reciprocity,
\begin{align*}
{\mathrm{Hom}}_{FG}(S_i, T{\uparrow}^G)\cong{\mathrm{Hom}}_{FN}({S_i}{\downarrow}_N, T)
&= {\mathrm{Hom}}_{FN}(F_N, T)=0
\text{ for  }i=0, 1
\\
{\mathrm{Hom}}_{FG}(S_2, T{\uparrow}^G)\cong{\mathrm{Hom}}_{FN}({S_2}{\downarrow}_N, T)
&\cong F \ \ \text{ as }F\text{-spaces}
\end{align*}
since $T\,{\not\cong}\,F_N$ 
(if $T\cong F_N$, then $S_2{\downarrow}_N \cong F_N$, so that $S_2$ is considered as a simple
$F\bar G$-module, and hence $F\bar G$ has three non-isomorphic simple modules,
a contradiction since $|\bar G|=2$).  Thus, by noting $1_B=1_b$ in (\ref{Dade99}) again, we have
$T{\uparrow}^G \cong S_2$, so that $\dim_F\,S_2 = 2\times\dim_F\,T$, a contradiction.

Hence, $S_2{\downarrow}_N \cong T\oplus T^g$.
So that, since $S_0=F_G, S_1, S_2$ are all non-isomorphic simple $FG$-modules in $B$
and since $1_B=1_b$, all non-isomorphic simple $FN$-modules in $b$ are
$T_0:=F_N, T, T^g$, and hence 
\begin{equation}\label{ell}  \ell(b)=3.  \end{equation}
Now, we have known that ${\mathrm{Irr}}(B)=\{\chi_0=1_G, \chi_1=\text{sign}, \chi_2, \chi_3\}$
and that $\chi_i{\downarrow}_N=\theta_0:=1_N$ for $i=0,1$.
Since $|\bar G|=2$ and since $1_B=1_b$ in (\ref{Dade99}), we obtain that
$$
\chi_2{\downarrow}_N = \theta_1 \text{ or }\theta_1+{\theta_1}^g
\text{ for some }\theta_1\in{\mathrm{Irr}}(b)-\{\theta_0\}
$$
and that
$$
\chi_3{\downarrow}_N = \theta_3 \text{ or }\theta_3+{\theta_3}^g
\text{ for some }\theta_3\in{\mathrm{Irr}}(b)-\{\theta_0\}.
$$
Hence, again by $1_B=1_b$ in (\ref{Dade99}), $k(b)\leq 5$.
Since $\ell(b)=3$ by (\ref{ell}), $k(b)\geq 4$. So that $k(b)=4$ or $5$.
Now recall that $p\,{\not=}\,3$ by Lemma \ref{nonCyclicElemAbe}.
Obviously, $p\,{\not=}\,2$ since $2=|G/N|\,\Big|\,|G/L|$. So that $p\geq 5$.
Moreover, since $b=B_0(F[L{\cdot}C_G(D)])$ and $b'=B_0(FL)$, 
\cite{Alp76, Dad77} implies that $b\cong b'$ as $F$-algebras.
So that 
\begin{equation}\label{Alperin-Dade} \qquad
k(b)=k(b') \text{ and }\ell(b)=\ell(b').
\end{equation}
Hence, if $k(b)=5$, then $k(b')=5$ and $\ell(b')=3$, so that Lemma \ref{k=5ell=3SimpleGp} yields that
we have a contradiction. Thus, $k(b)=4$. Namely
$$ k(b)=4 \text{ and }\ell(b)=3.
$$
Again by (\ref{Alperin-Dade}), $k(b')=4$ and $\ell(b')=3$. Then since $b'=B_0(FL)$, 
we finally have that $D\cong C_2\times C_2$ by Lemma \ref{k=4ell=3SimpleGp}(ii). 
\end{proof}

\section{What we can say if $k(B)=4$ and $\ell(B)=2$}

The purpose of this section is to prove the following theorem:

\begin{Theorem}\label{k=4ell=2MainTheorem}
Let $B:=B_0(FG)$, and let $D$ be a Sylow $p$-subgroup of $G$. 
If $k(B)=4$ and $\ell(B)=2$, then $|D|=5$.
\end{Theorem}

\noindent
In the following we list lemmas to give a complete proof of
Theorem \ref{k=4ell=2MainTheorem}. 

\begin{Notation}\label{H}
Throughout this section we assume that $G$ is a finite group with a Sylow
$p$-subgroup $D$ with $|D|=:p^d$ for some $d\geq 1$,
$B:=B_0(FG)$, and
we set $H:= N_G\Big(Z(J(D))\Big)$ where $J(D)$ is the Thompson subgroup of $D$,
and $\bar H:=H/O_{p'}(H)$.
\end{Notation}

\begin{Lemma}\label{k=4ell=2Normal_D}
If $k(B)=4$, $\ell(B)=2$ and $D\vartriangleleft G$, then
$|D|=5$ and $B\cong F\,D_{10}$ as $F$-algebras and also  
as interior $D$-algebras. 
\end{Lemma}

\begin{proof}
Since $B$ is principal, we can assume $O_{p'}(G)=1$. Then by Fong-Gasch{\"u}tz's
Theorem, $FG$ has only one block, say $B$. By Schur-Zassenhaus' Theorem,
$G=D\rtimes E$ for a $p'$-subgroup $E$ of $G$.
Thus, $k(G)=k(B)=4$ and $2=\ell(B)=\ell(G)=\ell(G/D)=\ell(E)=k(E)$.
Hence, \cite[Table 1]{LL85} implies the assertion.
\end{proof}

\begin{Hypothesis}\label{Hypo_k=4ell=2}
Suppose that $k(B)=4$ and $\ell(B)=2$ till just before the proof of Theorem \ref{k=4ell=2MainTheorem}.
Further, because of Lemma \ref{k=4ell=2Normal_D}, we assume that
Theorem \ref{k=4ell=2MainTheorem} holds for any finite group $\mathfrak G$ such that
$|\mathfrak G|<|G|$.
\end{Hypothesis}

\begin{Lemma}\label{G=O^{p'}(G)k=4ell=2}
We can assume that $G=O^{p'}(G)$.
\end{Lemma}

\begin{proof}
Set $L:=O^{p'}(G)$ and $b:=B_0(FL)$.
Suppose $\,G{\not=}\,L$.
Set $N:= L\,C_G(D)$.
If $N=G$, then \cite{Alp76, Dad77} implies that $k(B)=k(b)$ and $\ell(B)=\ell(b)$,
so that Hypothesis \ref{Hypo_k=4ell=2} yields that $|D|=5$ since $|L|<|G|$.
Hence we can assume $N\lneqq G$. Set $\bar G:=G/N$. 
Then, by Lemma \ref{1_B=1_b}(i)--(ii), it holds
$k(\bar G)\leq k(B)=4$. Since $\bar G\,{\not=}\,1$, $k(\bar G)\in\{ 2,3,4\}$.
If $k(\bar G)=4$, then $N\leq\bigcap_{\chi\in{\mathrm{Irr}}(B)}\,\ker(\chi) = O_{p'}(G)$
by \cite[Theorem 5.8.1]{NT89}, so $p\,{\not |}\, |N|$, and hence $p\,{\not |}\, |G|$, a contradiction.

Suppose $k(\bar G)=3$. Then, by \cite[Table 1]{LL85}, $\bar G = C_3$ or $\mathfrak S_3$.
Hence $\bar G$ has two distinct linear ordinary characters, so that $B$ has two distinct 
linear ordinary characters since ${\mathrm{Irr}}(\bar G)\subseteq {\mathrm{Irr}}(B)$,
say $\chi_0:=1_G$ and $\chi_1$,  by Lemma \ref{1_B=1_b}(i)--(ii).
Set $\phi_i:=(\chi_i)^0$ for $i=0,1$. Then, $\phi_0$ and $\phi_1$ are Brauer characters of $B$ and
note that $\phi_0\,{\not=}\,\phi_1$ as in the middle of the proof of Theorem \ref{k=4ell=3MainTheorem}.
Hence ${\mathrm{IBr}}(B)=\{\phi_0, \phi_1\}$ since $\ell(B)=2$. Namely, all irreducible Brauer characters
in $B$ are of degree one. Hence by \cite[Theorem $(5)\Leftrightarrow (10)$]{Kos81, Kos90},
$k(B)=k(\mathbb B)$ and $\ell(B)=\ell(\mathbb B)$, where $\mathbb B:=B_0(F\,N_G(D))$.
Thus, Lemma \ref{k=4ell=2Normal_D} yields the assertion.

Finally, suppose $k(\bar G)=2$. Then $\bar G\cong C_2$ by \cite[Table 1]{LL85}.
Thus as above, $B$ has two distinct linear ordinary characters, and hence 
we have the assertion.
\end{proof}

\begin{Lemma}\label{O^p(G)k=4ell=2}
If $O^p(G)\,{\lneqq}\,G$, then $D$ is abelian.
\end{Lemma}

\begin{proof}
First of all, by Lemma \ref{HZCp=2}, $D$ is abelian if $p=2$.
Hence we can assume that $p\,{\not=}\,2$.
Suppose that $D$ is non-abelian, and hence $d\geq 3$.

We want to have a contradiction.
By the assumption, $G$ has a normal subgroup $N\vartriangleleft G$ such that $|G/N|=p$.
Let $b:=B_0(FN)$. Then, $b$ is $G$-invariant, so that 
from Lemma \ref{Gallagher|G/N|=p} 
\begin{equation}\label{1B=1b}
1_B=1_b         \qquad\text{ and }
\end{equation}
\begin{equation}\label{Irr-G/N}
{\mathrm{Irr}}(B\,|\,{1_N}{\uparrow}^G) = {\mathrm{Irr}}(G/N).
\end{equation}
This means that
$B$ has $p$ distinct linear ordinary characters since $G/N\cong C_p$. 
That is, $p\leq|{\mathrm{Irr}}(B)|=k(B)$.
Then, since $k(B)=4$ 
and we are assuming $p\geq 3$, it holds that $p=3$.
Then, by \cite[Corollary 1.6]{Lan81}, $k_0(B)=3$ since $k(B)=4$.
On the other hand, since $G/N\cong C_3$, (\ref{Irr-G/N}) implies that
$B$ has three distinct linear ordinary characters. So that we can write 
${\mathrm{Irr}}(B)=\{\chi_0:=1_G, \chi_1, \chi_2, \chi_3\}$ such that 
$N\leq\ker{\chi_i}$ for $i=0,1,2$ and hence $\chi_0(1)=\chi_1(1)=\chi_2(1)=1$,
and that $\chi_3$ has positive height.
There is a character $\theta\in{\mathrm{Irr}}(b)$
such that $\langle{\chi_3}{\downarrow}_N, \, \theta\rangle^N\,{\not=}\,0$ by (\ref{1B=1b}).
Obviously, $\theta\,{\not=}\,1_N$ by (\ref{Irr-G/N}) since $\chi_3\,{\not\in}\,\{\chi_0, \chi_1, \chi_2\}$.

Assume first that $G_\theta=G$. Then, $\chi_3{\downarrow}_N=\theta$ since $G/N$ is cyclic.
So that $\langle\theta{\uparrow}^G,\,\chi_3\rangle^G=1$.
By the above,  $\langle\theta{\uparrow}^G,\,\chi_i\rangle^G=0$ for $i=0, 1, 2$.
Hence (\ref{1B=1b}) implies that ${\mathrm{Irr}}(b)=\{1_N, \theta\}$. Thus $k(b)=2$, so that
\cite[Theorem A]{Bra82} yields that $p=2$, a contradiction since $p=3$.
 
Thus $G_\theta\lneqq G$, and hence $\chi_3{\downarrow}_N=\theta+\theta^g+\theta^{g^2}$
where $G/N =\langle gN\rangle$ for an element $g\in G$.
This implies that ${\mathrm{Irr}}(b)=\{1_N, \theta, \theta^g, \theta^{g^2}\}$ by (\ref{1B=1b}).
Namely, $k(b)=4$.
Since $\ell(B)=2$, there is a simple $FG$-module $S$ in $B$ with $S\,{\not\cong}\,F_G$.
Since $G/N\cong C_p=C_3$ which is cyclic, we know by (\ref{1B=1b}) and Clifford's Theorem that
$$
S{\downarrow}_N \cong T\oplus T^g\oplus T^{g^2} \ \text{ or } \ 
S{\downarrow}_N \cong T 
\text{ for a simple }FN\text{-module }T\text{ in }b.
$$ 
If the first case occurs, then since $T\,{\not\cong}\,F_N$ 
(otherwise the multiplicity of $T$ in $S{\downarrow}_N$ as a direct summand
is not one, that contradicts the fact that $G/N$ is cyclic), 
$F_N, T, T^g, T^{g^2}$ are all non-isomorphic simple $FN$-modules in $b$ by (\ref{1B=1b}),
which yields that $\ell(b)\geq 4$, a contradiction since $k(b)=4$.
Hence only the second case happens. Namely, $S{\downarrow}_N \cong T$. Obviously, $N$ is not $3$-nilpotent
(otherwise $G$ is $3$-nilpotent since $|G/N|=3$, contradicting the fact that $\ell(B)=2$). Now, since
all non-isomorphic simple $FG$-modules in $B$ are $F_G$ and $S$,  
it follows again from (\ref{1B=1b}) that all non-isomorphic simple $FN$-modules in $b$ are $F_N$ and $T$.
This means that $\ell(b)=2$. Apparently, $|N|<|G|$. Hence Hypothesis \ref{Hypo_k=4ell=2} implies that
a Sylow $p$-subgroup of $N$ must be of order $5$, namely, $p=5$, a contradiction.
\end{proof}

\noindent
In fact the next lemma and Lemma \ref{k-ell=2} do work quite well
to prove Theorem \ref{k=4ell=2MainTheorem}.

\begin{Lemma}[\cite{BCR90}]\label{ZJ}
Assume that 
\begin{equation}\label{C_G(v)}  \
\text{$p$ is odd, and $C_G(u)$ is $p$-nilpotent for every element $u \in G$ of order $p$.}
\end{equation}
\begin{enumerate}
\item[\rm (i)] 
It holds that $H=N_G(D)$
and that $N_G(D)$ controls $p$-fusion in $G$, namely $\mathcal F_D(G)=\mathcal F_D(N_G(D))$.
\item[\rm (ii)]
Set $\mathbb B:=B_0(F\,N_G(D))$. Then,
the Scott module 
${\mathrm{Sc}}(G\times N_G(D),\,\Delta D)$ with vertex $\Delta D$ induces a splendid stable
equivalence of Morita type between $B$ and $\mathbb B$
where $\Delta D:=\{(u,u)\in G\times N_G(D)\,|\, u\in D\}$
(see \cite[Chap.5 \S 8]{NT89} for the Scott modules).
\end{enumerate}
\end{Lemma}

\begin{proof}
(i) If $D$ is abelian, then by Burnside's fusion theorem, $N_G(D)$ controls $p$-fusion in $G$ and
$J(D)=D$ by the definition of  Thompson's subgroup $J(D)$, and hence $H=N_G(D)$.

Thus we may assume that $D$ is non-abelian. 
Then, by Lemma \ref{O^p(G)k=4ell=2}, $O^p(G)=G$.
So that ${\mathrm{Ext}}^1_{FG}(F,F)={\mathrm{H}}^1(G,F)=0$ (see e.g. \cite[Chap.I Corollary 10.13]{Lan83}),
and hence $O^p(H)=H$ by Lemma \ref{ZJ_new} and \cite[Chap.I Corollary 10.13]{Lan83}.
Set $\bar H:=H/O_{p'}(H)$, and then $O^p(\bar H)=\bar H$. 
Now, since we may assume that $D\in{\mathrm{Syl}}_p(H)$ by Lemma \ref{ZJ_new}, 
it follows from (\ref{C_G(v)}) and 
\cite[Lemma 9.3]{BCR90} that $\bar H$ is a Frobenius group whose kernel is $\bar D:=(D\, O_{p'}(H))/O_{p'}(H) \cong D$. 
Hence $\bar D\vartriangleleft \bar H$, so that $D\vartriangleleft H$. This yields that $H\leq N_G(D)$.
On the other hand, since $Z(J(D)) \ {\mathrm{char}} \ J(D) \ {\mathrm{char}} \ D$, we know that $N_G(D)\leq H$.
Now, by (\ref{C_G(v)}) and \cite[the proof of Theorem 8.3]{BCR90}, we know that $G$ does not have the quadratic group ${\mathrm{Qd}}(p)$
in any subquotient of $G$. Hence well-known Glauberman's ZJ-theorem 
(see e.g. \cite[Theorem 1.21]{Cra11}) implies that
$\mathcal F_D(G)=\mathcal F_D(H)$. Therefore $\mathcal F_D(G)=\mathcal F_D(N_G(D))$ since we have proved that
$H=N_G(D)$.

(ii) 
Since we know $H=N_G(D)$ by (i), \cite[Theorem 8.3]{BCR90} implies the assertion.
\end{proof}

\begin{Lemma}\label{k-l=l(b_z)}
One of the following two cases happens:
\begin{enumerate}
\item[\rm (I)] $D$ is elementary abelian.
\item[\rm (II)] 
There are precisely {two} conjugacy classes of $G$ of non-trivial $p$-elements, and 
$C_G(u)$ is $p$-nilpotent for every element $u\in D^\#$.
\end{enumerate}
\end{Lemma}

\begin{proof} 
Since $k(B)-\ell(B)=2$, \cite[Theorems 5.9.4 and 5.6.1]{NT89} yields that the number of conjugacy classes of
$p$-elements in $G^\#$ is one or two.

So, assume first that the number is one. 
Then, $\ell(B_0(F[C_G(u)]))=2$ for every non-trivial $p$-element $u$ in $G$ 
by \cite[Theorems 5.9.4 and 5.6.1]{NT89}, and  all $p$-elements in $G^\#$ are $G$-conjugate.
If $G$ 
were one of the finite groups in \cite[(ii), (iii) of Theorem A]{KNST14},
then, by Lemma \ref{[KNST14,Theorem A]}, $k(B)-\ell(B)\,{\not=}\,2$, a contradiction.
Thus, by \cite[Theorem A]{KNST14}, $D$ is elementary abelian, so that Case (I) occurs.

Hence we can assume that the number is two, namely 
$k(B)-\ell(B)=2=\ell(b_u)+\ell(b_v)$ and $\ell(b_u)=\ell(b_v)=1$, where
$b_u:=B_0(F\,C_G(u))$ and $b_v:=B_0(F\,C_G(v))$ for two distinct elements $u, v\in D^\#$.
This implies that for every element $w\in D^\#$, $C_G(w)$ is $p$-nilpotent. So, Case II happens.
\end{proof}

\begin{Lemma}\label{k=4ell=2_abelianD}
We can assume that $D$ is abelian.
\end{Lemma}

\begin{proof}
First of all, the assertion holds for $p=2$ by Lemma \ref{HZCp=2}.
So, suppose that $p$ is odd.
Then, from Lemma \ref{k-l=l(b_z)}, it suffices to consider the case 
that the case (II) happens in   Lemma \ref{k-l=l(b_z)}.
So that for every element $u\in D^\#$, $C_G(u)$ is $p$-nilpotent. 

We may assume that $D$ is non-abelian. Set $L:=N_G(D)$ and $\mathbb B:=B_0(FL)$.
Hence it follows from Lemma \ref{ZJ}(ii) that $B$ and $\mathbb B$ are stably equivalent
of Morita type.
Thus, by \cite[Propositions 5.3 and 5.4]{Bro94}, it holds that $k(\mathbb B)-\ell(\mathbb B)=k(B)-\ell(B)$
and that $\bar Z(\mathbb B)\cong\bar Z(B)$ as $F$-algebras.
Namely, $k(\mathbb B)-\ell(\mathbb B)=2$. 
Set $\bar L:=L/O_{p'}(L)$.
Since $\mathbb B$ is the principal block of $FL$, 
$\mathbb B\cong B_0(F\bar L)$ as $F$-algebras and also as interior $D$-algebras.
Obviously, $C_{\bar L}(\bar w)$ is $p$-nilpotent for every element $\bar w\in\bar D^\#$ 
where $\bar D:=(D\,O_{p'}(L))/O_{p'}(L) \cong D$. 
Thus we can apply $\bar L$ to the group $G$ in Lemma \ref{k-ell=2}
(note that the action of $E$ on $\bar D$ is faithful where $E$ is a $p$-complement of $\bar L$).
Therefore, it follows from Lemma \ref{k-ell=2}(iv) that $\bar D$ is abelian, which means that
$D$ is abelian, a contradiction.
\end{proof}

\begin{Lemma}\label{elementaryAbelian}
We can assume that $D$ is elementary abelian.
\end{Lemma}

\begin{proof}
We can assume $O_{p'}(G)=1$ and $O^{p'}(G)=G$ by Lemma \ref{G=O^{p'}(G)k=4ell=2}.
Since $k(B)-\ell(B)=2$, it follows from the proof of Lemma \ref{k-l=l(b_z)} that
all non-identity $p$-elements are conjugate in $G$; or that there are precisely two conjugacy classes of
$G$ which have non-identity $p$-elements and $C_G(u)$ is $p$-nilpotent for every $u\in D^\#$.

If the first case occurs, then \cite[Theorem A]{KNST14} and Lemma \ref{[KNST14,Theorem A]}  imply that
$D$ is elementary abelian. Hence we can assume that the second case occurs.
Namely, there are precisely two conjugacy classes of
$G$ which have non-identity $p$-elements and $C_G(u)$ is $p$-nilpotent for every $u\in D^\#$.
Note that $D$ is abelian by Lemma \ref{k=4ell=2_abelianD}. 
Set $E:=N_G(D)/C_G(D)$.
Hence Burnside's fusion theorem says that
there are exactly three $E$-conjugacy classes of $D$, that is
$\dim_F\Big( (FD)^E  \Big)=3$ where $(FD)^E:=\{ a\in FD\,|\, y^{-1}ay=a, \ \forall y\in E \}$.
Now, by \cite[Proof of Theorem]{Oto19},
$$    (p^\epsilon +p-2)/(p-1) \leq LL\Big( (FD)^E\Big) \ \text{ where }p^\epsilon := {\mathrm{exp}}(D)$$
since $D$ is abelian. Clearly $LL\Big( (FD)^E\Big)\leq \dim_F\Big( (FD)^E  \Big)$, so that
$$ (p^\epsilon +p-2)/(p-1) \leq 3.$$ 
So $p^\epsilon+p-2\leq 3(p-1)$, and hence
if $\epsilon\geq 2$ then 
$0<(p-1)^2=p^2-2p+1\leq p^\epsilon-2p+1\leq 0$, a contradiction.
Hence $\epsilon=1$, which means that $D$ is elementary abelian.                                
\end{proof}

\begin{Lemma}\label{O_p(G)=1_k=4ell=2}
We can assume that $O_{p'}(G)=1$ and $O_p(G)=1$.
\end{Lemma}

\begin{proof}
Since we deal with only the principal block of $FG$, the first assertion is trivial.
Next, set $Q:=O_p(G)$, $\bar G:=G/Q$, $\bar D:=D/Q$ and $\bar B:=B_0(F\bar G)$.

Assume $Q\,{\not=}\,1$. It follows from \cite[Lemma 5.8.6]{NT89} that
$B$ dominates $\bar B$ and that ${\mathrm{Irr}}(\bar B)\subseteq{\mathrm{Irr}}(B)$.
Hence $k(\bar B)$ is $1$, $2$, $3$ or $4$.
If $k(\bar B)=1$, then $p\,{\not|}\,|\bar G|$, and hence $D=Q\vartriangleleft G$, so that the
assertion follows from Lemma \ref{k=4ell=2Normal_D}.
If $k(\bar B)=4$, then ${\mathrm{Irr}}(\bar B)={\mathrm{Irr}}(B)$, which means that
$Q\leq\bigcap_{\chi\in{\mathrm{Irr}}(B)}\,\ker \chi = O_{p'}(G)$, that is a contradiction since $Q\,{\not=}\,1$.

Next suppose that $k(\bar B)=2$. Then $|\bar D|=2$ by \cite[Theorem A]{Bra82}, 
so that $p=2$ and $\bar G$ is $2$-nilpotent by Burnside's Theorem. 
Hence $G$ is $2$-solvable, which yields that $G$ is of $2$-length one since $D$ is abelian 
by Lemma \ref{k=4ell=2_abelianD}. Now, since $O_{2'}(G)=1$, $D\vartriangleleft G$.
Hence we have a contradiction by  Lemma \ref{k=4ell=2Normal_D}. 

Finally suppose that $k(\bar B)=3$. Then Theorem \ref{MainTheorem} yields that $|\bar D|=3$, so that $p=3$.
Hence, by \cite[Corollary 1.6]{Lan81}, $k_0(B)=3$ since $k(B)=4$.
But this is a contradiction since we know that $D$ is abelian by  Lemma \ref{k=4ell=2_abelianD}
so that $k(B)=k_0(B)$ by \cite{KM13}. 
\end{proof}

\begin{Lemma}\label{reduction_k=4ell=2}
W can assume that $G$ is a non-abelian simple group with an abelian Sylow $p$-subgroup $D$ of $G$.
\end{Lemma}

\begin{proof}
By Lemmas \ref{O_p(G)=1_k=4ell=2} and \ref{G=O^{p'}(G)k=4ell=2}, 
we can assume $O_{p'}(G)=1=O_p(G)$ and  $G=O^{p'}(G)$.
Since 
$D$ is abelian by Lemma \ref{k=4ell=2_abelianD}, 
the same argument as the proof of Lemma \ref{k=5ell=3SimpleGp}(ii) yields the assertion.
\end{proof}

\noindent
Now we are ready to prove the main result of this section.

\begin{proof}[{\bf Proof of Theorem \ref{k=4ell=2MainTheorem}}]
First, assume that $D$ is cyclic. Then \cite{Dad66} implies that $4=k(B)=\ell(B)+m=2+m$, so that
$m=2$ where $m$ is the multiplicity of the exceptional characters in $B$, which implies that
$|D|=2{\cdot}2+1=5$, so we are done. Hence we can assume that $D$ is non-cyclic, so that $d\geq 2$.

Assume $p=2$. By Lemma \ref{k=4ell=2_abelianD}, $D$ is abelian. 
Let $\mathbb B:=B_0(F\,N_G(D))$. 
Then, by \cite[Theorem]{FH93}, 
$k(\mathbb B)=k(B)=4$ and $\ell (\mathbb B)=\ell(B)=2$. Thus, by Lemma \ref{k=4ell=2Normal_D},
$p=5$, a contradiction.

If $p=3$, then \cite[Corollary 1.6]{Lan81} implies that $k_0(B)=3$, so that $k_0(B)=3<4=k(B)$, that is a 
contradiction by \cite{KM13} since $D$ is abelian by  Lemma \ref{k=4ell=2_abelianD}.

Hence $p\geq 5$. By Lemmas \ref{elementaryAbelian} and \ref{reduction_k=4ell=2},
we can assume that $G$ is a 
non-abelian simple group with the non-cyclic elementary abelian Sylow $p$-subgroup $D$.

Then, just as in the proof of  Lemma \ref{k=5ell=3SimpleGp},      
it follows from Lemmas \ref{k=4ell=2Normal_D}, \ref{sporadic}, \ref{BroueMichel} 
and also 
the {\sf CFSG} that the assertion follows as before.
\end{proof}

\renewcommand{\baselinestretch}{1}
{\bf Acknowledgments.}
{\small 
The authors would like to thank Vyacheslav Aleksandrovich Belonogov and 
Benjamin Sambale for useful information;
and Radha Kessar for Lemma \ref{Radha}.
They are grateful also to Ron Brown, 
Marc Cabanes, Silvio Dolfi, Olivier Dudas, 
Klaus Lux,  Leo Margolis, Yoshihiro Otokita, and 
Atumi Watanabe for comments and advice.
Finally, the authors deeply would like to thank Gunter Malle who read the manuscript
of the first version and gave them many useful comments.
}


\begin{thebibliography}{99}

\bibitem{Alp76}
J.L.~Alperin,
\emph{Isomorphic blocks}, J.~Algebra {\bf 43} (1976), 694--698.

\bibitem{Alp87}
J.L.~Alperin,
\emph{Weights for finite groups}, Proc.~Symposia in Pure Math.  {\bf 47} (1987),
In ``The Arcata Conference on Representations of Finite Groups'', edited by P.~Fong, 
pp.369--379.

\bibitem{ACOU08}
J.~An, J.J.~Cannon, E.A.~O'Brien, W.R.~Unger,
\emph{The Alperin weight conjecture and Dade's conjecture for the simple group ${{\sf Fi}_{24}}'$},
LMS J.~Comput.~Math. {\bf 11} (2008), 100--145. 

\bibitem{AW04}
J.~An, R.A.~Wilson,
\emph{The Alperin weight conjecture and Uno's conjecture for the Baby Monster {\sf B}, $p$ odd},
LMS J.~Comput.~Math. {\bf 7} (2004), 120--166. 

\bibitem{AW10}
J.~An, R.A.~Wilson,
\emph{The Alperin weight conjecture and Uno's conjecture for the Monster {\sf M}, $p$ odd},
LMS J.~Comput.~Math. {\bf 13} (2010), 320--356. 

\bibitem{AKO11}
M.~Aschbacher, R.~Kessar, B.~Oliver,
\emph{Fusion Systems in Algebra and Topology},
London Math.~Soc.~Lecture Note Series {\bf 391}, Cambridge Univ.~Press, Cambridge, 2011.

\bibitem{Bel90}
V.A.~Belonogov,
\emph{Finite groups with a small principal $p$-block},
In: \emph{Group-theoretic investigations (in Russian)},
Akad.~Nauk SSSR Ural.~Otdel., Sverdlovsk, 1990, pp.8--30.

\bibitem{BCR90}
D.J.~Benson, J.F.~Carlson, G.R.~Robinson,
\emph{On the vanishing of group cohomology},
J.~Algebra {\bf 131} (1990), 40--73.

\bibitem{Bra82}
J.~Brandt,
\emph{A lower bound for the number of irreducible characters in a block},
J.~Algebra, {\bf 74} (1982), 509--575.

\bibitem{Bra63}
R.~Brauer,
\emph{Representations of Finite Groups},
In ``Lectures on Modern Mathematics, vol.I'',
Wiley, New York, 1963, pp.133--175.

\bibitem{BN41}
R.~Brauer, C.~Nesbitt,
\emph{On the modular characters of groups}, Ann.~of Math. {\bf 42} (1941), 556--590.

\bibitem{Bre94}
K.~Bremke,
\emph{The decomposition numbers of Hecke algebras of type $F_4$ with unequal parameters},
Manuscripta Math.~{\bf 83} (1994), 331--346.

\bibitem{Bro94}
M.~Brou\'e,
\emph{Equivalences of blocks of group algebras},
In ``Finite Dimensional Algebras and Related Topics'', edited by V.~Dlab and L.L.~Scott,
Kluwer Acad.~Pub., Dordrecht, 1994, 1-26.

\bibitem{BMM93}
M. Brou\'e, G.~Malle, J.~Michel,
\emph{Generic blocks of finite reductive groups}, 
In ``R\'epresentations unipotentes g\'en\'eriques et blocs de groupes r\'eductifs finis'',
Ast\'erisque Vol.{\bf 212} (1993), 7--92.

\bibitem{BM93}
M. Brou\'e, J.~Michel,
\emph{Blocs \`a groupes de d\'efaut ab\'elians des groupes r\'eductifs fini},
In ``R\'epresentations unipotentes g\'en\'eriques et blocs de groupes r\'eductifs finis'',
Ast\'erisque Vol.{\bf 212} (1993), 93--117.

\bibitem{Bur76}
R.~Burkhardt,
\emph{Die Zerlegungsmatrizen der Gruppen $PSL(2,p^f)$},
J.~Algebra {\bf 40} (1976), 75--96.

\bibitem{Cab18}
M.~Cabanes,
\emph{Local methods for blocks of finite simple groups},
In ``Local Representation 
Theory and Simple Groups''
(edited by R.~Kessar, G.~Malle, D.~Testerman),
EMS Series of Lectaures in Mathematics, European Math.~Soc., 
Z{\"u}rich, 2018, pp.179--265.

\bibitem{CE99}
M.~Cabanes, M.~Enguehard,
\emph{On blocks of finite reductive groups and twisted induction},
Adv.~Math. {\bf 145} (1999), 189--229.

\bibitem{CK92}
M.~Chlebowitz, B. K{\"u}lshammer,
\emph{Symmetric local algebras with $5$-dimensional center},
Trans.~Amer.~Math.~Soc. {\bf 329} (1992), 715--731.

\bibitem{Atlas} 
J.H.~Conway, R.T.~Curtis, S.P.~Norton, R.A.~Parker, R.A.~Wilson, 
\emph{Atlas of Finite Groups}, Clarendon Press, Oxford, 1985.

\bibitem{Cra11}
D.A.~Craven,
\emph{The Theory of Fusion Systems},
Cambridge Univ.~Press, Cambridge, 2011.

\bibitem{Dad66}
E.C.~Dade, \emph{Blocks with cyclic defect groups}, Ann.~of Math.  (2) {\bf 84}  (1966), 20--48.

\bibitem{Dad77}
E.C.~Dade, \emph{Remarks on isomorphic blocks},
J.~Algebra {\bf 45} (1977), 254--258.

\bibitem{Dad99}
E.C.~Dade, \emph{Another way to count characters},
J.~reine angew.~Math.  {\bf 510} (1999), 1--55.

\bibitem{FG61}
P.~Fong, W.~Gasch{\"u}tz,
\emph{A note on the modular representations of solvable groups},
J.~reine angew.~Math. {\bf 208} (1961), 73--78.

\bibitem{FH93}
P.~Fong, M.~Harris,
\emph{On perfect isometries and isotypies in finite groups},
Invent.~Math. {\bf 114} (1993), 139--191.

\bibitem{Fuj80}
M.~Fujii,
\emph{On determinants of Cartan matrices of $p$-blocks},
Proc.~Japan Acad. {\bf 55, Ser.A} (1980), 401--403.

\bibitem{GRSS19}
E.~Giannelli, N.~Rizo, B.~Sambale, A.A.~Schaeffer Fry,
\emph{Groups with few $p'$-character degrees in the principal block},
preprint, 2019.

\bibitem{Gor68}
D.~Gorenstein,
\emph{Finite Groups}, Harper and Row, New York, 1968.

\bibitem{GLS94}
D.~Gorenstein, R.~Lyons, R.~Solomon,
\emph{The Classification of the Finite Simple Groups, 
Number 1},
Math.~Surveys and Monograph, Vol.{\bf 40}, No.1,
Amer.~ Math.~Soc., 1994.

\bibitem{GLS98}
D.~Gorenstein, R.~Lyons, R.~Solomon,
\emph{The Classification of the Finite Simple Groups, 
Number 3},
Math.~Surveys and Monograph, Vol.{\bf 40}, No.3,
Amer.~ Math.~Soc., 1998.

\bibitem{Hen07}
S.~Hendren,
\emph{Extra special defect groups of order $p^3$ and exponent $p$},
J.~Algebra {\bf 313}  (2007), 724--760.

\bibitem{HW94}
G.~Hiss, D.L.~White,
\emph{
The $5$-modular characters of the covering group of the sporadic simple Fischer group 
${\sf Fi}_{22}$ and its automorphism group},
Comm.~Algebra {\bf 22} (1994), 3591--3611. 

\bibitem{Hum06}
J.E.~Humphreys,
\emph{Modular Representations of Finite Groups of Lie Type},
London Math.~Soc.~Lecture Note Series {\bf 326}, Cambridge Univ.~Press,
Cambridge, 2006.

\bibitem{Isa76}
I.M.~Isaacs, \emph{Character Theory of Finite Groups},
Academic Press, New York,  1976

\bibitem{Kar87}
G.~Karpilovsky,
\emph{The Schur Multiplier}, Oxford University Press, Oxford, 1987.

\bibitem{KL02}
R.~Kessar, M.~Linckelmann,
\emph{Blocks with Frobenius quotient},
J.~Algebra {\bf 249} (2002), 127--146.

\bibitem{KM13}
R.~Kessar, G.~Malle,
\emph{Quasi-isolated blocks and Brauer's height zero conjecture},
Ann.~of Math.  (2) {\bf 178}  (2013), 321--384.

\bibitem{KS95}
W.~Kimmerle, R.~Sandling,
\emph{Group-theoretic and group ring-theoretic determination of certain Sylow 
and Hall subgroups and the resolution of a question of R. Brauer},
J.~Algebra {\bf 171} (1995), 329--346. 

\bibitem{Kno76}
R. Kn{\"o}rr,
\emph{Blocks, vertices and normal subgroups},
Math.~Z. {\bf 148} (1976), 53--60.

\bibitem{Kos81}
S.~Koshitani, 
\emph{On the kernels of representations of finite groups}, 
Glasgow Math.~J.  {\bf 22} (1981), 151--154.

\bibitem{Kos90}
S.~Koshitani, 
\emph{On the kernels of representations of finite groups II}, 
Glasgow Math.~J.  {\bf 32} (1990), 341--347.

\bibitem{KK96}
S.~Koshitani, B. K{\"u}lshammer,
\emph{A splitting theorem for blocks},
Osaka J.~Math. {\bf 33} (1996), 343--346.

\bibitem{KS19}
S. Koshitani, T. Sakurai,
\emph{On theorems of Brauer-Nesbitt and Brandt for
characterizations of small block algebras},
Arch.~Math. {\bf 113 } (2019), 1--10.

\bibitem{Kue84}
B. K{\"u}lshammer,
\emph{Symmetric local algebras and small blocks of finite groups},
J.~Algebra {\bf 88} (1984), 190--195.

\bibitem{KNST14}
B. K{\"u}lshammer, G. Navarro, B. Sambale, P.H. Tiep,
\emph{Finite groups with 
two conjugacy classes of $p$-elements and related quotients for $p$-blocks},
Bull.~London Math.~Soc. {\bf 46} (2014), 305--314.

\bibitem{Lan81}
P.~Landrock,
\emph{On the number of irreducible characters in a $2$-block},
J.~Algebra {\bf 68} (1981), 426--442.

\bibitem{Lan83}
P.~Landrock,
\emph{Finite Group Algebras and their Modules},
London Math.~Soc.~Lecture Note Series {\bf 84}, Cambridge Univ.~Press, Cambridge, 1983.

\bibitem{Mal17}
G.~Malle,
\emph{Local-global conjectures in the representation theory of finite groups},
In ``Representation Theory -- Current Trends and Perspectives'',
EMS Ser. ~Congr. Rep., Eur.~Math.~Soc., Z{\"u}rich, 2017, pp.519--539.

\bibitem{MS16}
G.~Malle, B.~Sp{\"a}th, 
\emph{Characters of odd degree}, Ann.~of Math.  (2) {\bf 184} (2016), 869-908.

\bibitem{MT11}
G.~Malle, D.~Testerman,
\emph{Linear Algebraic Groups and Finite Groups of Lie Type},
Cambridge Univ.~Press, Cambridge, 2011.

\bibitem{MO91}
G.O.~Michler, J.B.~Olsson,
\emph{Weights for covering groups of symmetric and alternating groups, $p\, {\not=}\,2$},
Canad.~J.~Math. {\bf 43} (1991), 792--813.

\bibitem{NT89}
H.~Nagao, Y. Tsushima.
\emph{Representations of Finite Groups},
Academic Press, New York, 1989.

\bibitem{NT12}
G.~Navarro, P.H.~Tiep,
\emph{Brauer's height zero conjecture for the $2$-blocks of maximal defect},
J.~reine angew.~Math. {\bf 669}  (2012), 225--247.

\bibitem{NST18}
G.~Navarro, B.~Sambale, P.H.~Tiep,
\emph{Characters and Sylow $2$-subgroups of maximal class revisited},
J.~Pure Appl.~Algebra {\bf 222}  (2018), 3721--3732.

\bibitem{Oto19}
Y.~Otokita,
\emph{Lower bounds on Loewy lengths of centers of blocks},
to appear in Osaka J.~Math. (2019).

\bibitem{Sam14}
B.~Sambale,
\emph{Blocks of Finite Groups and Their Invariants},
Springer Lecture Notes in 
Mathematics Vol.{\bf 2127},
Springer, Heidelberg, 2014.

\bibitem{SZ16}
R.~Shen, Y.~Zhou,
\emph{Finite simple groups with some abelian Sylow subgroups},
Kuwait J.~Sci. {\bf 43} (2016), 1--15.

\bibitem{LL85}
A.~Vera L\'opez, J.~Vera L\'opez, 
\emph{Classification of finite groups according to the number of conjugacy classes},
Israel J. Math. {\bf 51}  (1985), 305--338. 

\bibitem{ModularAtlas}
R.~Wilson, J.~Thackray, R.~Parker, F.~Noeske,
J.~M{\"u}ller, F.~L{\"u}beck, C.~Jansen, G.~Hiss, T.~Breuer,
\emph{The Modular Atlas Project}, \quad
http://www.math.rwth-aachen.de/$^{\sim}$MOC.
\end{thebibliography}
\end{document}